\documentclass[11pt,a4paper]{article}
\usepackage{amsmath,amssymb,amsthm}
\usepackage[margin=2.7cm]{geometry}
\usepackage{enumitem}
\usepackage{tikz}
\usepackage[hidelinks]{hyperref}

\makeatletter
\renewcommand{\@seccntformat}[1]{\S\csname the#1\endcsname.\quad}
\renewcommand\section{\@startsection{section}{1}{\z@}%
  {-3.25ex\@plus -1ex \@minus -.2ex}%
  {1.5ex \@plus .2ex}%
  {\normalfont\large\bfseries}}
\makeatother

\theoremstyle{plain}
\newtheorem{lemma}{Lemma}
\newtheorem{proposition}[lemma]{Proposition}
\newtheorem*{theoremA}{Theorem A}
\newtheorem*{theoremB}{Theorem B}
\newtheorem*{corollaryC}{Corollary C}
\newtheorem*{corollaryD}{Corollary D}
\newtheorem*{theoremE}{Theorem E}
\theoremstyle{definition}
\newtheorem{remark}[lemma]{Remark}
\newtheorem{question}[lemma]{Question}

\newcommand{\vp}{\varpi}
\newcommand{\dl}{\mathrm{d}\ell}
\newcommand{\cG}{\mathcal{G}}
\newcommand{\rcod}{\rho_{\mathrm{cod}}}
\newcommand{\scod}{\sigma_{\mathrm{cod}}}
\DeclareMathOperator{\Syl}{Syl}
\DeclareMathOperator{\Irr}{Irr}
\DeclareMathOperator{\ccd}{ccd}
\DeclareMathOperator{\GF}{GF}

\newcommand{\drawF}[1]{%
\begin{tikzpicture}[x=1.45cm,y=1.05cm,
    vtx/.style={circle,draw,fill=white,inner sep=1.4pt}]
  \foreach \i in {1,...,#1} {
    \coordinate (t\i) at (\i,2); \coordinate (m\i) at (\i,1);
    \coordinate (b\i) at (\i,0);
    \draw (t\i)--(m\i)--(b\i);
    \node[font=\scriptsize] at (\i,2.38) {$\i$};
  }
  \pgfmathtruncatemacro{\last}{#1-1}
  \foreach \i in {1,...,\last} {
    \pgfmathtruncatemacro{\j}{\i+1}
    \ifodd\i
      \draw (t\i)--(m\j); \draw (m\i)--(b\j); \draw (b\i)--(t\j);
    \else
      \draw (t\i)--(b\j); \draw (m\i)--(t\j); \draw (b\i)--(m\j);
    \fi
  }
  \foreach \i in {1,...,#1} {
    \node[vtx] at (t\i) {}; \node[vtx] at (m\i) {}; \node[vtx] at (b\i) {};
  }
  \node[font=\small,anchor=east] at (0.6,2) {$t_i$};
  \node[font=\small,anchor=east] at (0.6,1) {$m_i$};
  \node[font=\small,anchor=east] at (0.6,0) {$b_i$};
\end{tikzpicture}}

\title{On Huppert's $\rho$--$\sigma$ conjecture for primes\\ in the orders of elements}
\author{Enrico Jabara \and Thomas Michael Keller}
\date{}

\begin{document}

\maketitle

\begin{abstract}
Let $G$ be a finite group, let $\rho(G)$ denote the number of distinct prime
divisors of $|G|$, and let $\sigma(G)$ be the maximal number of distinct prime
divisors of the order of an element of $G$. Let
$\rho(n)=\max\{\rho(G)\mid G \text{ solvable},\ \sigma(G)=n\}$.
We prove that $\rho(n)\le 4n-1+\pi(2n)$ for all $n\in\mathbb N$, and, with
the help of a computer, that $\rho(n)\le\frac92 n$ for all $n\le 16$. In
particular, $\rho(n)\le\lfloor\frac92 n\rfloor+2$ and
$\rho(n)\le\frac{101}{22}n$ for all $n$, and $\rho(n)\le\frac92 n$ for all
$n\ge 58$. This improves the bound $\rho(n)\le\frac{17}{3}n$ due to Yang, and
it shows that the bound $\rho(n)\le 5n$ stated by Bellotti, Keller and
Trudgian holds, although its original proof contains a gap. Moreover, we show
that $\rho(n)\ge 3n$ for all $n\ge 4$, which proves one half of a conjecture
of the second author stating that $\rho(n)=3n$ for $n\ge 4$. We also derive
a consequence for character codegrees and, conditionally, for arbitrary finite
groups. The
proofs of the main results do not rely on the classification of finite simple
groups.
\end{abstract}

\noindent\textbf{MSC 2020:} 20D60, 20E34, 20C15

\noindent\textbf{Keywords:} element orders; prime divisors;
$\rho$--$\sigma$ problem; $\sigma$-reduced groups; prime graph

\section{Introduction}\label{sec:intro}

All groups considered in this paper are assumed to be finite.

For a positive integer $m$ let $\vp(m)$ denote the set of prime divisors of
$m$, and for $x\in\mathbb R$ let $\pi(x)$ be the number of primes not
exceeding $x$. For a group $G$ we put $\vp(G):=\vp(|G|)$,
$\rho(G):=|\vp(G)|$,
\[
\sigma(G):=\max\bigl\{\,|\vp(o(g))| \bigm| g\in G\,\bigr\},
\]
and, for $n\in\mathbb N$,
\[
\rho(n):=\max\bigl\{\,\rho(G)\bigm| G \text{ solvable},\ \sigma(G)=n\,\bigr\}.
\]
The exact values of $\rho(n)$ are known only for $n\le 4$, and are summarized
below.

\begin{center}
\begin{tabular}{c|cccc}
\hline $n$ & 1 & 2 & 3 & 4\\ \hline
Known values of $\rho(n)$ & 2 & 5 & 8 & 12\\
References & Higman \cite{Hi57} & Zhang \cite{Zh95} & Keller \cite{Ke94}
 & Keller \cite{Ke95b}\\ \hline
\end{tabular}

\medskip
Table 1
\end{center}

As shown in \cite{Zh95}, the function $\rho(n)$ satisfies the quadratic bound
\[
\rho(n)\le\frac{n(n+3)}{2}\qquad\text{for all } n\in\mathbb N.
\]
This was later improved in \cite{Ke95b}, where the sharper estimate
\[
\rho(n)\le\frac{n(n+3)}{2}-2=\frac{(n-1)(n+4)}{2}\qquad\text{for all } n\ge 4
\]
was established. In \cite{Ke95a} it is shown that $\rho(n)\le Cn\ln n$ for a
constant $C>0$, and the first linear bound was obtained in \cite{Ke95c}: there
it is proved that $\rho(n)\le 4n+\pi(\rho(n))$, whence
$\rho(n)\le(4+\varepsilon)n$ for every $\varepsilon>0$ and all sufficiently
large $n$. An explicit linear bound valid for all $n$ is due to Yang, who
proved in \cite{Ya21} that $\rho(n)\le\frac{17}{3}n$. In \cite{BKT23},
Bellotti, Keller and Trudgian state that $\rho(n)\le 5n$ for all $n$; their
argument, however, contains a gap, which we analyze in \S\ref{sec:digression}.

The general problem of determining the optimal constant $K$ such that
\[
\rho(n)\le Kn\qquad\text{for all } n\in\mathbb N
\]
remains open. In analogy with Huppert's classical $\rho$--$\sigma$ conjecture
concerning character degrees, it is conjectured that $\rho(n)=3n$ for all
$n\ge 4$, so that $K=3$; this was first explicitly formulated in 1995 as
Conjecture~9 in \cite{Ke95c}.

The first purpose of this note is to prove the following result.

\begin{theoremA}
For all $n\in\mathbb N$ we have
\[
\rho(n)\le 4n-1+\pi(2n).
\]
In particular, $\rho(n)\le\lfloor\frac92 n\rfloor+2$ for all
$n\in\mathbb N$, and $\rho(n)\le\frac92 n$ for
$n\in\{1,2,44,46,47,\dots,54,56\}$ and for all $n\ge 58$.
\end{theoremA}

Theorem A is derived from Proposition~\ref{prop:key2}, which shows that
$\rho(G)\le 4\sigma(G)-1$ for every solvable group $G\ne 1$ all of whose
prime divisors exceed $2\sigma(G)$. Its proof relies on the method of
\cite{Ke95c}, in the form of a chain lemma (Lemma~\ref{lem:chain}), and we
take the opportunity to correct some minor slips in \cite[Lemma~5]{Ke95c}
(see Remark~\ref{rem:slips}). For $n\ge 8$, Theorem A is better than the
previously known bounds (see Table~2 and Remark~\ref{rem:largen} in
\S\ref{sec:proofs}).

For small $n$ we obtain better bounds by combining these methods with lower
bounds for the ranks of the Sylow subgroups of $\sigma$-reduced groups. This
leads to a finite optimization problem, which we solve by computer (see
\S\ref{sec:small}).

\begin{theoremB}[computer-assisted]
	For $6\le n\le 17$, the following bounds hold.
	\begin{center}
		\begin{tabular}{c|cccccccccccc}
			\hline $n$ & 6 & 7 & 8 & 9 & 10 & 11 & 12 & 13 & 14 & 15 & 16 & 17\\ \hline
			$\rho(n)\le$ & 24 & 29 & 33 & 38 & 43 & 48 & 54 & 58 & 63 & 67 & 72 & 77\\
			$\lfloor 9n/2\rfloor$ & 27 & 31 & 36 & 40 & 45 & 49 & 54 & 58 & 63 & 67 & 72 & 76\\ \hline
		\end{tabular}
	\end{center}
\end{theoremB}

For $n=6$, Theorem B improves the bound $\rho(6)\le 25$ of \cite{Ke95b}. Together with Table~1 and the bound $\rho(5)\le 18$ of \cite{Ke95b}, Theorems A and B show that $\rho(n)\le\frac92 n$ for all $n$, except possibly for $17\le n\le 43$ and $n\in\{45,55,57\}$; for these $n$, our bound exceeds $\lfloor\frac92 n\rfloor$ by one or two (see Remark~\ref{rem:range}, where we also discuss these remaining cases). Moreover, these results imply that $\rho(n)\le\frac{101}{22}n$ for all $n\in\mathbb N$, the worst case being $n=22$, where Theorem A gives $\rho(22)\le 101$; Theorem A together with Table~1 gives $\rho(n)\le\frac{33}{7}n$.

As observed by Moret\'o in \cite{Mo21}, obtaining a good estimate for the constant $K$ has valuable consequences in related problems. In particular, let
$\chi\in\Irr(G)$ be an irreducible character of a group $G$, and define its
codegree as $\ccd(\chi)=|G:\ker\chi|/\chi(1)$. The associated invariants are
\[
\scod(G)=\max_{\chi\in\Irr(G)}\bigl|\vp(\ccd(\chi))\bigr|,\qquad
\rcod(G)=\Bigl|\bigcup_{\chi\in\Irr(G)}\vp(\ccd(\chi))\Bigr|.
\]
In \cite{Mo21}, Moret\'o proves that $\rcod(G)\le\frac32\scod(G)$ if $G$ has
trivial solvable radical, and he derives that for every group $G$
\[
\rcod(G)\le\Bigl(K+\frac32\Bigr)\scod(G),
\]
where $K$ is any real number such that $\rho(n)\le Kn$ for all
$n\in\mathbb N$. Combining this with Theorems A and B, we immediately obtain
the following consequence.

\begin{corollaryC}
For every group $G$ we have $\rcod(G)\le\frac{67}{11}\,\scod(G)$.
\end{corollaryC}

\noindent Without Theorem B, one obtains the constant $\frac{87}{14}$ in the
same way.

The function $\rho_g(n)$ is defined analogously to $\rho(n)$, except that
arbitrary finite groups are considered, not necessarily solvable. Following
the approach of \cite{Mo06} and \cite{HY21} (see also
\cite[Corollary~1]{BKT23}), we obtain the following conditional result. Its
hypothesis, and its relation to the results of \cite{HY21} and \cite{BKT23},
are discussed in Remark~\ref{rem:HY}.

\begin{corollaryD}
Suppose that $\rho(T)\le 8\bigl(\sigma(T)+1\bigr)^3+1$ for every direct
product $T$ of nonabelian simple groups. Then $\rho_g(n)\le 70\,n^4$ for
every $n\in\mathbb N$.
\end{corollaryD}

Concerning lower bounds, by \cite[Remark~1.11]{Ke94} we have
$\rho(n)\ge 2n$ for all $n$, and $\rho(n)\ge 3n$ whenever $n$ is divisible
by $4$; the latter is obtained by taking direct products of groups as
constructed in \cite[Example~1.9]{Ke94}. As observed in \cite[\S3]{BKT23},
in order to obtain $\rho(n)\ge 3n$ for all $n\ge 4$ it would therefore
suffice to construct suitable groups for $n=5,6,7$. Using the
characterization of the prime graphs of solvable groups due to Gruber,
Keller, Lewis, Naughton and Strasser \cite{GKLNS15}, we show that the
conjectured value $3n$ is indeed a lower bound for all $n\ge 4$; this proves
one half of \cite[Conjecture~9]{Ke95c}.

\begin{theoremE}
For every $n\ge 4$ we have $\rho(n)\ge 3n$. In particular,
$\rho(n)\ge 3n-1$ for all $n\in\mathbb N$, with equality for $n\le 3$.
\end{theoremE}

Together with Theorem A, we obtain $3n\le\rho(n)\le 4n-1+\pi(2n)$ for all
$n\ge 4$. Theorem E is proved in \S\ref{sec:lower} by means of an explicit
family of graphs (see Figure~\ref{fig:F}).

We point out that the proofs of Theorems A, B and E do not rely on the
classification of finite simple groups; Corollaries C and D use results on
nonsolvable groups from \cite{Mo06}, \cite{Mo21} and \cite{HY21}.

\section{Notation and preliminary results}\label{sec:prelim}

We recall some notions from \cite{Ke95c} (see also \cite{Ke94}). A group $G$
is called \emph{$\sigma$-reduced} if $G=P_1\cdots P_m$, where
$P_i\in\Syl_{p_i}(G)$ for pairwise distinct primes $p_1,\dots,p_m$, such that
$P_i\le N_G(P_j)$ for $i<j$ and each $P_i$ is isomorphic to a principal factor
of $G$; we put $M_G:=\{P_1,\dots,P_m\}$. For groups $U$ and $X$ we write
$U\lesssim X$ if $U$ is isomorphic to a subgroup of $X$. We shall use the
following facts \cite[Lemma~3]{Ke95c}.

\begin{enumerate}[label=(\alph*)]
\item Every solvable group $G$ possesses a $\sigma$-reduced subgroup $U$ with
$\vp(U)=\vp(G)$; clearly $\sigma(U)\le\sigma(G)$.
\item If $G$ is $\sigma$-reduced and $P,Q\in M_G$ with $P\le N_G(Q)$, then
$P$ acts either faithfully or trivially on $Q$.
\item If $G$ is $\sigma$-reduced, then
\[
G^{(i)}/G^{(i+1)}\cong P_{i,1}\times\cdots\times P_{i,n_i}
\qquad(0\le i<\dl(G)),
\]
where $P_{i,l}\in M_G$ for all $i,l$. Furthermore, if $0\le j<k<\dl(G)$ and
$P\in M_G$ with $P\lesssim G^{(k)}/G^{(k+1)}$, then there exists a $Q\in M_G$
with $Q\lesssim G^{(j)}/G^{(j+1)}$ such that $Q$ acts faithfully on $P$.
\end{enumerate}

Let $G$ be $\sigma$-reduced. By (c), all Sylow subgroups of $G$ are
elementary abelian, and the sets $\vp(G^{(i)}/G^{(i+1)})$,
$0\le i<\dl(G)$, are pairwise disjoint. Following \cite{Ke95c}, we write
\[
n_i=n_i(G):=\bigl|\vp(G^{(i)}/G^{(i+1)})\bigr|\qquad(i\ge 0),
\]
so that $n_i=0$ for $i\ge\dl(G)$, and we put
\[
\widehat n_i:=\sum_{j\ge i}n_j=\rho(G^{(i)}),\qquad
\nu_i:=\sum_{t\ge 0}n_{i+2t}\qquad(i\ge 0).
\]
Thus $\widehat n_i=\nu_i+\nu_{i+1}$ and $\rho(G)=n_0+n_1+\nu_2+\nu_3$. As
$G^{(i)}/G^{(i+1)}$ is abelian, it contains an element whose order is
divisible by all primes in $\vp(G^{(i)}/G^{(i+1)})$, and so does every
preimage of this element in $G$; hence $n_i\le\sigma(G)$ for all $i$. If
$\sigma(G)=n$, we put $\lambda_i:=n-n_i\ge 0$. Moreover, $G/G'$ is cyclic of
squarefree order: if $p\in\vp(G/G')$ and $P\in\Syl_p(G)$, then, as $P$ is
isomorphic to a principal factor of $G$, a chief series of $G$ through $G'$
has exactly one factor of $p$-power order, and this factor lies above $G'$;
thus $|P|=p$.

For a group $X$ all of whose Sylow subgroups are elementary abelian we put,
as in \cite{Ke95c},
\[
e(X):=\min\bigl\{\dim_{\GF(p)}P\bigm| p \text{ prime},\ 1<P\in\Syl_p(X)\bigr\}.
\]
Let $G$ be $\sigma$-reduced. We say that $P\in M_G$ \emph{lies in the $i$-th
layer} of $G$ if $P\lesssim G^{(i)}/G^{(i+1)}$; thus $n_i$ is the number of
members of $M_G$ in the $i$-th layer, and for $0\le i<\dl(G)$ we put
\[
e_i:=\min\bigl\{e(P)\bigm| P\in M_G \text{ lies in the $i$-th layer}\bigr\}.
\]
If $P,P'\in M_G$ lie in the $i$-th and the $j$-th layer, respectively, where
$i\le j$, then $P$ normalizes $P'$: either $P\le N_G(P')$ by definition, or
$P'\le N_G(P)$, and then $[P,P']\le P\cap G^{(j)}=1$ if $i<j$, while for
$i=j$ we have $[P,P']\le PP'\cap G^{(i+1)}=1$. In particular, members in the
same layer commute.

For a solvable group $X$ we denote by $h(X)$ its Fitting height, and by
$h^*(X)\le h(X)$ the invariant introduced by Turull \cite{Tu84} by means of
reduced Fitting chains.

For $N\in\mathbb N_0$ let $\cG_N$ denote the class of solvable groups $G$
such that every prime divisor of $|G|$ is greater than $N$, and for
$n\in\mathbb N$ put
\[
B_N(n):=\max\bigl\{\rho(H)\bigm| H\in\cG_N\ \sigma\text{-reduced},\
\sigma(H)\le n\bigr\}.
\]
Let $G$ be a solvable group with $\sigma(G)=n$, and let $K$ be a Hall
subgroup of $G$ for the set of primes greater than $N$. By (a), $K$
possesses a $\sigma$-reduced subgroup $H$ with $\vp(H)=\vp(K)$; clearly
$H\in\cG_N$ and $\sigma(H)\le n$. As at most $\pi(N)$ primes divide
$|G:K|$, we conclude that
\begin{equation}\label{eq:reduction}
\rho(n)\le B_N(n)+\pi(N)\qquad\text{for all } n,N\in\mathbb N.
\end{equation}

\begin{lemma}\label{lem:centralizer}
Let $\Gamma$ be a group with $\sigma(\Gamma)=n$, and let $x\in\Gamma$. If
$h\in C_\Gamma(x)$ and $(o(h),o(x))=1$, then
$|\vp(o(h))|\le n-|\vp(o(x))|$. In particular,
$\sigma(H)\le n-|\vp(o(x))|$ for every subgroup $H\le C_\Gamma(x)$ with
$(|H|,o(x))=1$.
\end{lemma}

\begin{proof}
As $x$ and $h$ commute and have coprime orders, we have $o(xh)=o(x)o(h)$,
and hence $|\vp(o(x))|+|\vp(o(h))|=|\vp(o(xh))|\le n$.
\end{proof}

\begin{lemma}\label{lem:dl}
Let $\Gamma$ be a $\sigma$-reduced group with $\sigma(\Gamma)=n$. Then
\[
\dl(\Gamma)\le n+1+\lambda_0\qquad\text{and}\qquad
\dl(\Gamma)\le n+2+\lambda_1 .
\]
In particular:
\begin{enumerate}[label=(\arabic*)]
\item If $n_0=n$, then $\dl(\Gamma)\le n+1$.
\item If $n_0=n-1$, then $\dl(\Gamma)\le n+2$.
\item If $n_0=n-2$, then $\dl(\Gamma)\le n+3$.
\item If $n_1=n$, then $\dl(\Gamma)\le n+2$.
\item If $n_1=n-1$, then $\dl(\Gamma)\le n+3$.
\end{enumerate}
Moreover, $\dl(\Gamma)\le n+3$ whenever $\lambda_0+\lambda_1\le 3$.
\end{lemma}

\begin{proof}
By (c), for every $i\ge 0$ the members of $M_\Gamma$ contained in
$\Gamma^{(i)}$ provide a Fitting chain of $\Gamma^{(i)}$ in the sense of
\cite{Tu84} of length $\dl(\Gamma)-i$ (it is not even necessary to take
quotients), and this chain is reduced, since the factors
$\Gamma^{(j)}/\Gamma^{(j+1)}$ have pairwise coprime orders. As the Sylow
subgroups of $\Gamma^{(i)}$ are abelian, its derived length equals its
Fitting height, and therefore
$h^*(\Gamma^{(i)})=h(\Gamma^{(i)})=\dl(\Gamma)-i$. Recall also that
$\dl(C)\le 2\sigma(C)$ for every group $C$ all of whose Sylow subgroups are
elementary abelian \cite[Lemma~3(c)]{Ke95a}.

We may assume that $\Gamma'\ne1$. Let $A$ be a complement of $\Gamma'$ in
$\Gamma$. By the remarks above, $A=\langle a\rangle$ is cyclic of squarefree
order, and $|\vp(o(a))|=n_0$ is the composition length of $A$. The group $A$
acts coprimely on $\Gamma'$, and we put $C:=C_{\Gamma'}(A)$. By
\cite[Corollary~3.2]{Tu84} (cf.\ the proof of \cite[Lemma~3(c)]{Ke95a}) we have
$h^*(\Gamma')\le n_0+h^*(C)$. By Lemma~\ref{lem:centralizer},
$\sigma(C)\le n-n_0=\lambda_0$, and the Sylow subgroups of $C$ are
elementary abelian, so that $h^*(C)\le h(C)\le\dl(C)\le 2\lambda_0$. Hence
\[
\dl(\Gamma)=1+h^*(\Gamma')\le 1+n_0+2\lambda_0=n+1+\lambda_0 .
\]
Similarly, we may assume that $\Gamma''\ne 1$, and we let $B$ be a
complement of $\Gamma''$ in $\Gamma'$. Then $B$ is abelian with elementary
abelian Sylow subgroups, so it contains an element $b$ of squarefree order
with $|\vp(o(b))|=n_1$. The group $\langle b\rangle$ acts coprimely on
$\Gamma''$, and with $C:=C_{\Gamma''}(b)$ we obtain as before, using
\cite[Corollary~3.2]{Tu84} and Lemma~\ref{lem:centralizer},
\[
\dl(\Gamma)=2+h^*(\Gamma'')\le 2+n_1+h^*(C)\le 2+n_1+2\lambda_1
=n+2+\lambda_1 .
\]
The remaining assertions follow immediately; for the last one note that
$\lambda_0\le 2$ or $\lambda_1=0$ if $\lambda_0+\lambda_1\le 3$.
\end{proof}

\section{General bounds}\label{sec:general}

The main results of \cite{Ke95c} rest on Lemmas~4--6 there. Lemma~4 of
\cite{Ke95c} states the following. Let $G$ be $\sigma$-reduced, let $s\ge 0$,
and let $U$ be an abelian subgroup of $G$ with $U\lesssim G/G^{(s+1)}$ such
that $e(P_0)\ge t+1$ for every Sylow subgroup $P_0$ of $U$, where
$P_0\le P\in M_G$ and $t$ is the number of members of $M_G$ contained in
$G^{(s)}$ on which $P$ acts faithfully. Then there exist an $x\in U$ with
$|\vp(o(x))|=\rho(U)$ and a $\sigma$-reduced subgroup
$H\le C_G(x)\cap G^{(s)}$ with $\dl(H)=\dl(G)-s$ and $n_i(H)=n_{i+s}(G)$ for
all $i$. Lemma~5 of \cite{Ke95c} iterates this construction, and
\cite[Lemma~6]{Ke95c} is the following consequence; it is the only result of
\cite{Ke95c} that we shall need.

\begin{lemma}[{\cite[Lemma~6]{Ke95c}}]\label{lem:Ke6}
Let $l\ge 2$ and $k\ge 0$ be integers, and let $G$ be a $\sigma$-reduced
group with the following property: whenever $j\ge k$ and $l$ divides $j-k$,
every prime in $\vp(G^{(j)}/G^{(j+1)})$ is greater than
$\rho(G^{(j+l+1)})$. Then
\[
\sum_{i\ge 1}n_{k+li}\le\sigma(G).
\]
\end{lemma}

In \cite{Ke95c} this is formulated for the class $X(l,k)$ of all
$\sigma$-reduced groups with the stated property.

\begin{remark}\label{rem:slips}
We take this opportunity to point out some slips in \cite[Lemma~5]{Ke95c};
none of them affects the results of \cite{Ke95c}. In part~(a), the inductive
construction should run over $j\in\{1,\dots,m\}$ and end with $G_{m+1}=1$; as
written, $x_m$ is not constructed, and $G_m\ne1$ since $\dl(G_m)=k_m$. In
part~(b), the function $f_i$ is evaluated at the $k_i+1$ numbers
$n_{h_{i-1}}(G),\dots,n_{h_i}(G)$, so it should be a map
$\mathbb N^{k_i+1}\to\mathbb N$, as in the proof of \cite[Lemma~6]{Ke95c}. In
the proof of~(b), $G$ should be replaced by $G/G^{(h_m+1)}$ rather than by
$G/G^{(\dl(G)-h_m-1)}$; in fact, this reduction is not needed at all. The
induction on $m$ applies the assertion to $H$ with the sequence
$(k_2,\dots,k_m)$, which formally requires hypothesis~(1) for the partial
sums of this sequence. This is harmless, since the groups constructed in the
later steps lie in $G^{(h_j)}$ and have the properties required in~(1)
relative to $G$ itself; moreover, in \cite[Lemma~6]{Ke95c} all $k_i$ are
equal to $l$ anyway. Finally, the proof of \cite[Lemma~6]{Ke95c} does not
verify hypothesis~(1) of \cite[Lemma~5(b)]{Ke95c} for the class $X(l,0)$.
It holds because the conditions in~(1) force $\vp(H^{(i)})=\vp(G^{(N+i)})$
for all $i$ (as $H^{(i)}\le G^{(N+i)}$, and both groups have
$\sum_{t\ge i}n_{t+N}(G)$ prime divisors), so that every prime divisor of
$|H|$ lies in the layer of $H$ corresponding to its layer in $G$, shifted
by $N$.
\end{remark}

The following general bound is essentially \cite[Theorem~7]{Ke95c}.

\begin{proposition}\label{prop:general}
Let $G$ be a $\sigma$-reduced group with $\sigma(G)=n$. Then
\[
\rho(G)\le 4n+\pi(\rho(G''')).
\]
In particular, $\rho(n)\le 4n+\pi(\rho(n))$ for all $n\in\mathbb N$.
\end{proposition}

\begin{proof}
Put $m:=\rho(G''')$, let $K$ be a Hall subgroup of $G$ for the set of primes
greater than $m$, and let $L\le K$ be a $\sigma$-reduced subgroup with
$\vp(L)=\vp(K)$, which exists by (a). For all $j\ge0$ we have
$L^{(j+3)}\le L'''\le G'''$, and hence $\rho(L^{(j+3)})\le m<p$ for every
$p\in\vp(L)$. Therefore Lemma~\ref{lem:Ke6} applies to $L$ with $l=2$ and
$k\in\{0,1\}$, and yields $\nu_2(L)\le\sigma(L)$ and $\nu_3(L)\le\sigma(L)$,
where $\nu_i(L)$ is defined as above with respect to $L$. As
$n_0(L),n_1(L)\le\sigma(L)\le n$, we obtain $\rho(L)\le 4n$, and so
$\rho(G)\le\rho(L)+\pi(m)\le 4n+\pi(m)$.

For the second assertion, let $G$ be a solvable group with $\sigma(G)=n$ and
$\rho(G)=\rho(n)$, and apply the first one to a $\sigma$-reduced subgroup $U$
as in (a), noting that $\sigma(U)\le n$ and $\rho(U''')\le\rho(U)=\rho(n)$.
\end{proof}

\section{A digression}\label{sec:digression}

In the paper \cite{BKT23} by Bellotti, Keller and Trudgian, it is stated that
$\rho(n)\le 5n$ for every $n$. However, the proof given there contains a gap,
which we now describe.

Let $\Gamma$ be a $\sigma$-reduced group with $\sigma(\Gamma)=n$, and put
$\lambda:=\lambda_0+\lambda_1+\lambda_2$. An essential step of the argument
in \cite{BKT23} is the estimate
\begin{equation}\label{eq:BKT}
\rho(\Gamma)\le 4n-\lambda+\pi(\rho(\Gamma''')).
\end{equation}
Since $\rho(\Gamma)=n_0+n_1+n_2+\rho(\Gamma''')=3n-\lambda+\rho(\Gamma''')$,
the estimate \eqref{eq:BKT} is equivalent to
\[
\rho(\Gamma''')\le n+\pi(\rho(\Gamma''')),
\]
that is, to a bound of the form $\widehat n_3\le n$, up to small primes. No
such bound is available. The methods of \cite{Ke95c} control the layers of
even and of odd index separately: they yield Proposition~\ref{prop:general},
which is \eqref{eq:BKT} without the term $-\lambda$, and, for groups all of
whose prime divisors exceed $2n$, the estimates $\nu_3\le n$ and $\nu_4\le n$
of Proposition~\ref{prop:key} below, which give only
$\rho(\Gamma''')=\nu_3+\nu_4\le 2n$.

After the first author brought this to the attention of the authors of
\cite{BKT23} (one of whom is the second author of the present paper), they
revisited the argument and found that their methods only yield
$\rho(n)\le\frac{21}{4}n$ (personal communication). Nevertheless, the bound
$\rho(n)\le 5n$ stated in \cite{BKT23} is true: it is a consequence of
Theorem A, since $\pi(2n)\le n$ for all $n$.

This discussion highlights a major obstacle to improving the constant~$4$: by
Proposition~\ref{prop:key2}, we have $\rho(G)\le 4n-1-\lambda_0-\lambda_1$ for
every $\sigma$-reduced group $G\in\cG_{2n}$ with $\sigma(G)=n$, so that one
needs control over the quantity $n_0+n_1$. A small opening is provided by
Lemma~\ref{lem:dl}, which bounds the derived length of $G$ when $\lambda_0$
or $\lambda_1$ is small (see Remark~\ref{rem:range}).

\section{A key tool}\label{sec:key}

The key ingredients of the proof of Theorem A are the following consequence of
Lemma~\ref{lem:Ke6} and its refinement Proposition~\ref{prop:key2} below,
which is based on a chain lemma (Lemma~\ref{lem:chain}).

\begin{proposition}\label{prop:key}
Let $G$ be a $\sigma$-reduced group with $\sigma(G)=n$, and suppose that every
prime divisor of $|G|$ is greater than $2n$. Then $\nu_j\le n$ for every
$j\ge 2$. In particular,
\[
\sum_{i\ge 1}n_{2i}\le n,\qquad \sum_{i\ge 1}n_{2i+1}\le n,\qquad
\rho(G'')\le 2n,
\]
and $\rho(G)\le 4n-\lambda_0-\lambda_1\le 4n$.
\end{proposition}

\begin{proof}
Since $\nu_j=0$ for $j\ge\dl(G)$, we may argue by downward induction on $j$.
Let $j\ge 2$, and suppose that $\nu_{j+1}\le n$ and $\nu_{j+2}\le n$. Then
\[
\rho(G^{(j+1)})=\widehat n_{j+1}=\nu_{j+1}+\nu_{j+2}\le 2n .
\]
Let $t\ge j-2$ and $p\in\vp(G^{(t)}/G^{(t+1)})$. As $t+3\ge j+1$, we have
$\rho(G^{(t+3)})\le\rho(G^{(j+1)})\le 2n<p$. Hence the hypothesis of
Lemma~\ref{lem:Ke6} is satisfied with $l=2$ and $k=j-2$, and
Lemma~\ref{lem:Ke6} yields
\[
\nu_j=\sum_{i\ge 1}n_{j-2+2i}\le\sigma(G)=n .
\]
This completes the induction. The remaining assertions follow, since
\[
\sum_{i\ge 1}n_{2i}=\nu_2,\qquad \sum_{i\ge 1}n_{2i+1}=\nu_3,\qquad
\rho(G'')=\nu_2+\nu_3,\qquad \rho(G)=n_0+n_1+\nu_2+\nu_3 .
\]
\end{proof}

In order to refine Proposition~\ref{prop:key}, we first note that the ranks
of the members of $M_G$ outside the first two layers are bounded below by the
smallest prime divisor of $|G|$; this is the argument of step $(+)$ in the
proof of \cite[Lemma~6]{Ke95c}.

\begin{lemma}\label{lem:rankp}
Let $G$ be a $\sigma$-reduced group all of whose prime divisors are at least
$p$. Then $e_k\ge p$ for every $k$ with $2\le k<\dl(G)$.
\end{lemma}

\begin{proof}
Let $P\in M_G$ lie in the $k$-th layer. By (c) there are $R\in M_G$ in the
$(k-1)$-th layer acting faithfully on $P$, and $Q\in M_G$ in the $(k-2)$-th
layer acting faithfully on $R$. Let $1\ne y\in Q$. As $R$ is abelian and
$o(y)$ is prime to $|R|$, we have $R=[R,y]\times C_R(y)$ with $[R,y]\ne 1$ and
$C_{[R,y]}(y)=1$. Hence $F:=\langle y\rangle[R,y]$ is a Frobenius group with
cyclic complement $\langle y\rangle$ and elementary abelian kernel $[R,y]$.
As $Q$ and $R$ normalize $P$, the group $F$ acts on $P$, and faithfully so,
since $[R,y]$ acts faithfully on $P$ and every nontrivial normal subgroup of
$F$ meets $[R,y]$ nontrivially. Therefore $e(P)\ge o(y)\ge p$ by
\cite[Lemma~3(d)]{Ke95c}.
\end{proof}

\begin{lemma}[Chain lemma]\label{lem:chain}
Let $G$ be a $\sigma$-reduced group with $\sigma(G)\le n$, all of whose prime
divisors are at least $p$, and let $2\le s_1<s_2<\dots<s_m$ be integers with
$s_{k+1}-s_k\ge 2$ for $1\le k<m$. If $\rho(G^{(s_1+1)})\le p-1$, then
\[
n_{s_1}+\dots+n_{s_m}\le n,\qquad\text{and}\qquad
n_{s_1}+\dots+n_{s_m}+n_t\le n\quad\text{for every } t>s_m .
\]
\end{lemma}

\begin{proof}
We may assume that $s_m<\dl(G)$. Put $s_0:=0$ and $H_0:=G$. Let $1\le k\le m$,
and suppose that we have already found commuting elements
$x_1,\dots,x_{k-1}$ of $G$ and a $\sigma$-reduced subgroup
$H_{k-1}\le C_G(x_1)\cap\dots\cap C_G(x_{k-1})$ with
$H_{k-1}\le G^{(s_{k-1})}$, $\dl(H_{k-1})=\dl(G)-s_{k-1}$ and
$n_i(H_{k-1})=n_{s_{k-1}+i}(G)$ for all $i$. As in Remark~\ref{rem:slips},
these conditions force $\vp(H_{k-1}^{(i)})=\vp(G^{(s_{k-1}+i)})$ for all $i$,
so that the primes in the $i$-th layer of $H_{k-1}$ are exactly those in the
$(s_{k-1}+i)$-th layer of $G$. Put $j:=s_k-s_{k-1}\ge 2$, and let $U$ be the
product of the members of $M_{H_{k-1}}$ in the $j$-th layer of $H_{k-1}$;
then $U$ is abelian. A member $P$ of this layer acts faithfully only on
members of $M_{H_{k-1}}$ in later layers, and their number is at most
$\rho(H_{k-1}^{(j+1)})=\rho(G^{(s_k+1)})\le\rho(G^{(s_1+1)})\le p-1$, while
$e(P)\ge p$ by Lemma~\ref{lem:rankp}, applied to $H_{k-1}$. Hence
\cite[Lemma~4]{Ke95c} (see \S\ref{sec:general}) applies to $H_{k-1}$, with $s=j$
and this $U$. It yields an element $x_k\in U$ whose order is divisible by
exactly the $n_{s_k}$ primes of the $s_k$-th layer of $G$, and a
$\sigma$-reduced subgroup $H_k\le C_{H_{k-1}}(x_k)\cap H_{k-1}^{(j)}$ with
$\dl(H_k)=\dl(G)-s_k$ and $n_i(H_k)=n_{s_k+i}(G)$ for all $i$. This completes
the inductive construction.

The elements $x_1,\dots,x_m$ commute pairwise and have pairwise coprime
orders, as they involve the primes of different layers of $G$. Hence
$x:=x_1\cdots x_m$ satisfies $|\vp(o(x))|=n_{s_1}+\dots+n_{s_m}\le\sigma(G)$,
which is the first inequality. Now let $s_m<t<\dl(G)$. The $(t-s_m)$-th
layer of $H_m$ involves exactly the $n_t$ primes of the $t$-th layer of $G$,
so (as in \S\ref{sec:prelim}) there is an element $w\in H_m^{(t-s_m)}$ whose
order is divisible by these $n_t$ primes. All prime divisors of $o(w)$ lie in
the layers $t,t+1,\dots$ of $G$, so that $o(w)$ is prime to $o(x)$. As
$w\in H_m\le C_G(x)$, Lemma~\ref{lem:centralizer} yields
$n_t\le|\vp(o(w))|\le\sigma(G)-|\vp(o(x))|$, which is the second
inequality.
\end{proof}

\begin{proposition}\label{prop:key2}
Let $G$ be a $\sigma$-reduced group with $\sigma(G)=n$, and suppose that every
prime divisor of $|G|$ is greater than $2n$. Then $\rho(G'')\le 2n-1$, and
hence $\rho(G)\le 4n-1-\lambda_0-\lambda_1$.
\end{proposition}

\begin{proof}
Let $p$ be the smallest prime divisor of $|G|$. By
Proposition~\ref{prop:key}, $\rho(G^{(j)})=\nu_j+\nu_{j+1}\le 2n\le p-1$ for all
$j\ge 2$, so Lemma~\ref{lem:chain} applies to every sequence with $s_1\ge 2$.
Let $K:=\dl(G)-1$. If $K\le 2$, then $\rho(G'')\le n_2\le n$. Let $K\ge 3$, so
that $n_K\ge 1$. If $K$ is odd, then Lemma~\ref{lem:chain}, applied to the
sequence $2,4,\dots,K-1$ and $t=K$, gives $\nu_2+n_K\le n$, so that
$\nu_2\le n-1$, while $\nu_3\le n$ by Proposition~\ref{prop:key}. If $K$ is
even, the sequence $3,5,\dots,K-1$ and $t=K$ give $\nu_3\le n-1$, while
$\nu_2\le n$. In both cases $\rho(G'')=\nu_2+\nu_3\le 2n-1$, and therefore
$\rho(G)=n_0+n_1+\rho(G'')\le 4n-1-\lambda_0-\lambda_1$.
\end{proof}

Compared with Proposition~\ref{prop:key}, the gain comes from the extra layer
$t$ in Lemma~\ref{lem:chain}. It is available because \cite[Lemma~4]{Ke95c}
provides not only the element $x$, but a whole $\sigma$-reduced subgroup of
its centralizer. In particular, the configuration $n_0=n_1=\nu_2=\nu_3=n$ does
not occur among the groups considered in Proposition~\ref{prop:key2}.

\begin{remark}\label{rem:key}
By (a), Proposition~\ref{prop:key2} implies that $\rho(G)\le 4\sigma(G)-1$
for every solvable group $G\ne 1$ all of whose prime divisors exceed
$2\sigma(G)$. We
do not know whether it suffices to assume that every prime divisor of $|G|$
exceeds $\sigma(G)$ (see Question~\ref{q:Gn}). The point is that
\cite[Lemma~4]{Ke95c} requires, for the Sylow subgroups of $U$, a rank
exceeding the number of members of $M_G$ below the current layer on which
they act faithfully, and in the proofs of \cite[Lemma~6]{Ke95c} and of
Lemma~\ref{lem:chain} this rank is provided by the size of the prime divisors
of $|G|$ (via Lemma~\ref{lem:rankp}). The following fixed
point lemma requires much less, but it does not suffice to run the argument
of \cite{Ke95c} (see Remark~\ref{rem:fixed}).
\end{remark}

\begin{lemma}\label{lem:fixed}
Let $P$ be an elementary abelian $p$-group of rank $e$ acting on a group $H$,
and assume that the following conditions are satisfied:
\begin{enumerate}[label=(\arabic*)]
\item $p$ does not divide $|H|$;
\item all Sylow subgroups of $H$ are elementary abelian;
\item $H$ has a Sylow tower, and $H'$ is a Hall subgroup of $H$;
\item $P$ acts faithfully on every $P$-invariant Sylow subgroup of $H$;
\item if $H_r$ and $H_q$ are Sylow subgroups of $H$ for distinct primes $r$
and $q$ with $H_r\le N_H(H_q)$, then $[H_r,H_q]=1$ or $H_r$ acts faithfully
on $H_q$;
\item $|\vp(H/H')|<e$.
\end{enumerate}
Then $\vp(H')\subseteq\vp(C_H(y))$ for every $1\ne y\in P$, and there exists
an element $1\ne x\in P$ such that $\vp(C_H(x))=\vp(H)$.
\end{lemma}

\begin{proof}
Put $L:=H'$. By (3), $H$ is solvable, and so there exists a $P$-invariant
Sylow system $\{H_q\mid q\in\vp(H)\}$ of $H$ (see, e.g., \cite{Hu67}). By (3)
we may write $\vp(H)=\{p_1,\dots,p_k\}$ such that
$N_i:=H_{p_i}H_{p_{i+1}}\cdots H_{p_k}$ is normal in $H$ for $1\le i\le k$;
put $N_{k+1}:=1$. Then $H_{p_i}\le N_H(H_{p_j})$ for $i<j$, since
$H_{p_j}=H_{p_i}H_{p_j}\cap N_j$.

Let $r=p_j\in\vp(L)$. We claim that $[H_{p_i},H_r]\ne 1$ for some $i<j$.
Otherwise $N_j/N_{j+1}$ is central in $H/N_{j+1}$, so that $H/N_{j+1}$ is
the direct product of $N_j/N_{j+1}$ and the image of the Hall subgroup
$H_{p_1}\cdots H_{p_{j-1}}$. But then $(H/N_{j+1})'$ is an $r'$-group, and
as $N_{j+1}$ is an $r'$-group, so is $L$, a contradiction.

So let $i<j$ with $[H_{p_i},H_r]\ne 1$, and put $Q:=H_{p_i}$. By (5), $Q$
acts faithfully on $H_r$, and by (4), $P$ acts faithfully on $Q$. Let
$1\ne y\in P$, and suppose that $C_{H_r}(y)=1$. As $Q$ is abelian and
$P$-invariant, we have $Q=[Q,y]\times C_Q(y)$, so that $C_{[Q,y]}(y)=1$.
Since $[Q,y]$ normalizes $H_r$, the group $M:=[Q,y]H_r$ is $y$-invariant, and
$C_M(y)H_r/H_r=C_{M/H_r}(y)=1$ by coprime action; hence $y$ acts fixed point
freely on $M$. As $y$ has prime order, $M$ is nilpotent by Thompson's theorem
\cite{Th59}, whence $[[Q,y],H_r]=1$. Since $Q$ acts faithfully on $H_r$, we
get $[Q,y]=1$, i.e., $y$ centralizes $Q$, contradicting (4). Therefore
$C_{H_r}(y)\ne 1$, and $r\in\vp(C_H(y))$. This proves the first assertion.

Now let $V_1,\dots,V_s$ be the Sylow subgroups of $H/L$; thus
$s=|\vp(H/L)|<e$ by (6). Each $V_i$ is a $P$-invariant elementary
abelian group of order prime to $p$. Let $\overline{\mathbb F}$ be an
algebraic closure of the prime field of $V_i$. Then $P$ has a common
eigenvector on $V_i\otimes\overline{\mathbb F}$, and so there is a subgroup
$K_i\le P$ of index at most $p$ with $C_{V_i}(K_i)\ne 1$. Put
$D:=K_1\cap\cdots\cap K_s$. Then $|D|\ge p^{e-s}>1$, and every $1\ne x\in D$
satisfies $C_{V_i}(x)\ne 1$ for all $i$. As $C_{H/L}(x)=C_H(x)L/L$ by
coprime action, we obtain $\vp(H/L)\subseteq\vp(C_H(x))$. Since
$\vp(H)=\vp(H/L)\cup\vp(L)$, the lemma follows.
\end{proof}

\begin{remark}\label{rem:fixed}
(a) Hypothesis (4) of Lemma~\ref{lem:fixed} does not follow from the other
hypotheses, even if $H=G''$ for a $\sigma$-reduced group $G$ and
$P\in M_G$ with $P\lesssim G'/G''$. To see this, let $A$ be cyclic of order
$3$, acting irreducibly on $P\cong C_5\times C_5$ and nontrivially on
$B\cong C_7$, let $Q$ be a faithful irreducible $\GF(11)[AB]$-module, and let
$K:=A(P\times B)Q$, where $P$ acts trivially on $Q$. Then $N:=P\times Q$ is
an abelian normal subgroup of $K$, and $A$ permutes the six subgroups of
order $5$ of $P$ in two orbits $O_1$ and $O_2$. For $j=1,2$ choose
$1\ne\alpha_j\in\Irr(P)$ with $\ker\alpha_j\in O_j$, and choose
$1\ne\beta\in\Irr(Q)$. The inertia group of $\alpha_j\times\beta$ in $K$ is
$N$, since $A$ moves $\alpha_j$ and $B$ acts fixed point freely on
$\Irr(Q)$. Hence $V_j:=(\alpha_j\times\beta)^K$, realized over $\GF(r_j)$
with $r_1=331$ and $r_2=661$ (note that $r_j\equiv 1\bmod 55$), is a
faithful irreducible $K$-module of dimension $21$. The group
$G:=K\ltimes(V_1\times V_2)$ is $\sigma$-reduced, the sets
$\vp(G^{(i)}/G^{(i+1)})$ are $\{3\}$, $\{5,7\}$, $\{11\}$ and $\{331,661\}$
for $i=0,1,2,3$, and $H:=G''=Q(V_1\times V_2)$ satisfies (1)--(3), (5) and
(6) with respect to $P$ (here $e=2$), but not (4). An element $1\ne y\in P$
centralizes a nontrivial element of $V_j$ if and only if
$\langle y\rangle\in O_j$. As $O_1\cap O_2=\emptyset$, we have
$\vp(C_H(y))\ne\vp(H)$ for every $1\ne y\in P$.

(b) Lemma~\ref{lem:fixed} provides an element of prime order only, and no
$\sigma$-reduced subgroup of its centralizer with the layer structure
described in \cite[Lemma~4(2),(3)]{Ke95c}. These are exactly the features
needed to iterate the construction as in \cite[Lemma~5]{Ke95c}, and
therefore Lemma~\ref{lem:fixed} cannot replace \cite[Lemma~4]{Ke95c} in the
proofs of Proposition~\ref{prop:key} and Lemma~\ref{lem:chain}.
\end{remark}

Finally, we recall the following counting lemma from \cite{Ke95a}; see
Remark~\ref{rem:range} for its relevance.

\begin{lemma}[{\cite[Lemma~5]{Ke95a}}]\label{lem:counting}
Let $H$ be a group all of whose Sylow subgroups are elementary abelian, and
suppose that $H=U_1U_2\cdots U_m$, where each $U_k$ is an abelian product of
Sylow subgroups of $H$, the sets $\vp(U_k)$ are pairwise disjoint, and
$U_k\le N_H(U_l)$ for $k<l$. Put $s_k:=\rho(U_k)$, and assume that
\[
e(U_k)\ge\prod_{l=k+1}^{m}(s_l+1)\qquad\text{for } k=1,\dots,m.
\]
Then $\sigma(H)=\rho(H)$.
\end{lemma}

\section{Proofs of Theorem A and Corollary C}\label{sec:proofs}

\begin{proof}[Proof of Theorem A]
Let $n\in\mathbb N$, and let $1\ne H\in\cG_{2n}$ be $\sigma$-reduced with
$m:=\sigma(H)\le n$. Every prime divisor of $|H|$ is greater than
$2n\ge 2m$, so Proposition~\ref{prop:key2} yields $\rho(H)\le 4m-1\le 4n-1$.
Hence $B_{2n}(n)\le 4n-1$, and \eqref{eq:reduction} with $N=2n$ gives
\[
\rho(n)\le 4n-1+\pi(2n).
\]

Next we show that $\pi(2n)\le\frac n2$ for $n\ge 60$. If $2n\ge 152$, this
follows from the estimate $\pi(x)<1.25506\,x/\ln x$ for $x>1$
\cite[Corollary~1]{RS62}, as $\ln 152>5.023>4\cdot 1.25506$. If
$120\le 2n\le 150$, we use that $\pi(120)=30$ and that the primes in the
interval $(120,150]$ are $127,131,137,139,149$, with
$\pi(127)=31$, $\pi(131)=32$, $\pi(137)=33$, $\pi(139)=34$ and
$\pi(149)=35$; since $\pi$ is constant between consecutive primes, this shows
$\pi(x)\le\frac x4$ for $120\le x\le 150$. For $n<60$, a direct check shows
that $\pi(2n)-\frac n2\le 3$, with equality exactly for
$n\in\{10,12,16,22,24\}$, and that $\pi(2n)\le\frac n2+1$ holds exactly for
$n\in\{1,2,44,46,47,\dots,54,56,58,59\}$. Hence
$\rho(n)\le 4n-1+\pi(2n)\le\frac92 n+2$ for all $n$, and
$\rho(n)\le\frac92 n$ for the values of $n$ stated in Theorem A.
\end{proof}

For $n\ge 8$, Theorem A is better than the bounds of \cite{Ke95b} and
\cite{Ya21}, and for $6\le n\le 17$ it is improved further by Theorem B. For
some small values of $n$ the bounds are compared in the following table.

\begin{center}
	\begin{tabular}{c|ccccc}
		\hline $n$ & 8 & 9 & 10 & 11 & 12\\ \hline
		$(n-1)(n+4)/2$, see \cite{Ke95b} & 42 & 52 & 63 & 75 & 88\\
		$\lfloor 17n/3\rfloor$, see \cite{Ya21} & 45 & 51 & 56 & 62 & 68\\
		Theorem A, $4n-1+\pi(2n)$ & 37 & 42 & 47 & 51 & 56\\
		Theorem B & 33 & 38 & 43 & 48 & 54\\ \hline
	\end{tabular}
	
	\medskip
	Table 2. Upper bounds for $\rho(n)$.
\end{center}

\begin{remark}\label{rem:largen}
Since $\pi(x)<1.25506\,x/\ln x$ for $x>1$ \cite[Corollary~1]{RS62},
Theorem A yields
\[
\rho(n)\le\Bigl(4+\frac{2.51012}{\ln(2n)}\Bigr)n-1\qquad\text{for all }
n\in\mathbb N .
\]
Combined with the exact values of $\pi(x)$ for $x\le 10^6$, this shows, for
instance, that $\rho(n)\le 4.4n$ for $n\ge 163$, $\rho(n)\le 4.3n$ for
$n\ge 1084$, $\rho(n)\le 4.25n$ for $n\ge 4200$ and $\rho(n)\le 4.2n$ for
$n\ge 32\,335$. On the other hand, as $\pi(x)>x/\ln x$ for $x\ge 17$
\cite[Corollary~1]{RS62}, we have $4n-1+\pi(2n)>4.1n$ for
$9\le n\le 2.4\cdot 10^8$. In particular, Theorem A does not yield the bound
$\rho(n)\le 4.1n$ for $n\ge 1000$ mentioned in \cite{BKT23}, which relies on
the argument discussed in \S\ref{sec:digression}.
\end{remark}

\begin{proof}[Proof of Corollary C]
As observed by Moret\'o in \cite{Mo21}, for every group $G$ we have
\[
\rcod(G)\le\Bigl(K+\frac32\Bigr)\scod(G),
\]
where $K$ is any real number such that $\rho(n)\le Kn$ for all
$n\in\mathbb N$. We show that $K=\frac{101}{22}$ is admissible, which yields
$\rcod(G)\le\bigl(\frac{101}{22}+\frac32\bigr)\scod(G)=\frac{67}{11}\scod(G)$.
For $n\le 5$ this follows from Table~1 and the bound $\rho(5)\le 18$ of
\cite{Ke95b}, and for $6\le n\le 17$ from Theorem B. For
$n\ge 76$ we have $\pi(2n)<\frac n2$ (see the proof of Theorem A), so that
$\rho(n)<\frac92 n$ by Theorem A, and for $18\le n\le 75$ a direct check shows
that $4n-1+\pi(2n)\le\frac{101}{22}n$, with equality only for $n=22$. In the
same way, $K=\frac{33}{7}$ is admissible without the use of Theorem B: for
$n\le 7$ this follows from Table~1 and the bound
$\rho(n)\le\frac{(n-1)(n+4)}{2}$ of \cite{Ke95b}, which gives
$\rho(7)\le 33$, and for $8\le n\le 75$ a direct check shows that
$4n-1+\pi(2n)\le\frac{33}{7}n$. This gives the constant
$\frac{33}{7}+\frac32=\frac{87}{14}$.
\end{proof}

\section{Small values of $n$}\label{sec:small}

In this section we prove Theorem B. Given $n$, we choose a prime $p$ and
apply \eqref{eq:reduction} with $N=p-1$; then we have to bound $\rho(G)$ for
$\sigma$-reduced groups $G$ with $\sigma(G)\le n$ all of whose prime divisors
are at least $p$. If $p\le 2n$, then Proposition~\ref{prop:key} is not
available, and Lemma~\ref{lem:chain} applies only to some sequences; we
combine it with lower bounds for the ranks $e_k$ and with
Lemma~\ref{lem:counting}.

For a prime $q$ let $r(q)$ denote the smallest prime congruent to $1$ modulo
$q$, and for a prime $p$ put
\[
\mathfrak q(p):=\min\bigl(\{r(q)\mid q\ge p\text{ prime}\}\cup\{p^2\}\bigr).
\]
For instance, $\mathfrak q(11)=23$,
$\mathfrak q(13)=\mathfrak q(17)=\mathfrak q(19)=\mathfrak q(23)=47$,
$\mathfrak q(29)=59$ and $\mathfrak q(31)=\mathfrak q(37)=83$. Note that
$2p+1\le\mathfrak q(p)\le p^2$ for $p\ge 3$, and that $r\ge\mathfrak q(p)$ for
every prime $r$ with $r\equiv 1\pmod q$ for some prime $q\ge p$. In the
following, $p^+$ denotes the smallest prime greater than $p$.

\begin{lemma}\label{lem:orbit}
Let $X$ be a group acting transitively on a finite set $\Omega$, let $R$ be an
elementary abelian normal $r$-subgroup of $X$ acting faithfully on $\Omega$,
and let $Q\le X$ be a subgroup of prime order $q\ne r$ which acts
nontrivially on $R$. Then $|\Omega|\ge\min\{qr,r^o\}$, where $o$ is the
order of $r$ modulo $q$. If, moreover, $q$ and $r$ are at least $p\ge 3$, then
$|\Omega|\ge\mathfrak q(p)$.
\end{lemma}

\begin{proof}
As $R$ is normal in $X$ and $X$ is transitive, the $R$-orbits on $\Omega$ form
a system of blocks for $X$, and they all have the same length, $r^a$ say. As
$R\ne 1$ acts faithfully, $a\ge 1$. Let $t$ be the number of $R$-orbits, so
that $|\Omega|=tr^a$. The group $Q$ permutes the $R$-orbits. If $Q$ moves one
of them, then $t\ge q$ and $|\Omega|\ge qr$. Otherwise, for every $R$-orbit
$O$, the group $Q$ normalizes the kernel $K_O$ of $R$ on $O$ and acts on the
elementary abelian group $R/K_O$, which acts regularly on $O$ and therefore
has order $r^a$. If $Q$ acted trivially on $R/K_O$ for every $O$, then
$[R,Q]$ would act trivially on $\Omega$, whence $[R,Q]=1$, a contradiction.
Hence $R/K_O$ is a nontrivial $\GF(r)Q$-module for some $O$, so that $a\ge o$
and $|\Omega|\ge r^o$.

For the last assertion, note that $qr\ge pp^+>p^2\ge\mathfrak q(p)$, as
$q\ne r$; if $o=1$, then $r\equiv 1\pmod q$, so that $r\ge\mathfrak q(p)$; and
if $o\ge 2$, then $r^o\ge p^2\ge\mathfrak q(p)$.
\end{proof}

\begin{lemma}\label{lem:ranks}
Let $G$ be a $\sigma$-reduced group all of whose prime divisors are at least
$p\ge 3$, and let $K:=\dl(G)-1$. Then:
\begin{enumerate}[label=(\arabic*)]
\item $e_{k-1}\le e_k$ for $1\le k\le K$;
\item $e_k\ge p^{(k-1)/2}$ for $1\le k\le K$;
\item $e_2\ge p$, $e_3\ge\mathfrak q(p)$ and $e_4\ge p\,\mathfrak q(p)$,
provided that $K\ge 2$, $K\ge 3$ and $K\ge 4$, respectively.
\end{enumerate}
\end{lemma}

\begin{proof}
(1) Let $P\in M_G$ lie in the $k$-th layer. By (c), some $Q\in M_G$ in the
$(k-1)$-th layer acts faithfully on $P$, so that $e(P)\ge e(Q)\ge e_{k-1}$ by
\cite[Lemma~3(e)]{Ke95c}.

(2) This is \cite[Lemma~7(b)]{Ke95a}.

(3) The first assertion is Lemma~\ref{lem:rankp}. Let $k\in\{3,4\}$ with
$k\le K$, let $W\in M_G$ lie in the $k$-th layer, let $\ell$ be the prime
with $W\in\Syl_\ell(G)$, and let $V$ be the Sylow $\ell$-subgroup of
$G^{(k)}/G^{(k+1)}$. Then $V\cong W$, and $V$ is a $\GF(\ell)G$-module on
which $G^{(k)}$ acts trivially; the chief factors of $G$ of $\ell$-power order
are the composition factors of $V$, and as $W$ is isomorphic to one of them,
$V$ is irreducible. By (c) there is a $Y\in M_G$ in the $(k-1)$-th layer
acting faithfully on $W$; then $\overline Y:=YG^{(k)}/G^{(k)}$ is normal in
$G/G^{(k)}$ and acts faithfully on $V$. Let $\overline{\mathbb F}$ be an
algebraic closure of $\GF(\ell)$, and let $\lambda\ne 1$ be a linear
character of $\overline Y$ whose eigenspace on $V\otimes\overline{\mathbb F}$
is nonzero. The group $X:=G/G^{(k-1)}$ acts on the characters of
$\overline Y$, and all $X$-conjugates of $\lambda$ occur in
$V\otimes\overline{\mathbb F}$. Let $S$ be the stabilizer of $\lambda$ in $X$,
and let $M:=G^{(k-2)}/G^{(k-1)}$. Then
\[
e(W)=\dim V\ge|X:S|=|X:SM|\,|M:M\cap S| .
\]
Here $M$ is the direct product of its Sylow subgroups $M_j$, which are the
images of the members $P_j$ of $M_G$ in the $(k-2)$-th layer, and as for $V$
we see that every $M_j$ is an irreducible $X$-module (under conjugation);
moreover, $M\cap S$ is the direct product of the groups $M_j\cap S$. The
$X$-conjugates of $\lambda$ generate an $X$-invariant subgroup of the
character group of $\overline Y$, which is an irreducible $X$-module, being
dual to $\overline Y$; hence they generate the whole character group. As some
$M_j$ acts nontrivially on $\overline Y$ by (c), it follows that $M\not\le S$,
that is, $I:=|M:M\cap S|>1$. Next, $|X:SM|$ divides $|X:M|$, so that
$|X:SM|=1$ or $|X:SM|\ge p$. If $SM=X$, then $M\cap S$ is normal in $X$, so
that every $M_j\cap S$ is a normal subgroup of $X$ contained in $M_j$, and
hence $M_j\cap S\in\{1,M_j\}$ for all $j$.

Let $k=3$. If $I$ is not a prime, then $e(W)\ge I\ge p^2\ge\mathfrak q(p)$. If
$I=m$ is a prime and $SM\ne X$, then $e(W)\ge pm\ge p^2\ge\mathfrak q(p)$.
Finally, if $I=m$ is a prime and $SM=X$, then $M_j\cap S=1$ for the unique $j$
with $M_j\not\le S$, so that $|P_j|=m$. By (c), some member of the $0$-th
layer, of prime order $q\ge p$, acts faithfully on $P_j$, so that
$m\equiv 1\pmod q$ and $e(W)\ge m\ge\mathfrak q(p)$.

Let $k=4$; then $e(P_j)\ge p$ for all $j$ by Lemma~\ref{lem:rankp}. If
$SM=X$, then $I\ge|P_j|\ge p^p\ge p^3\ge p\,\mathfrak q(p)$ for every $j$
with $M_j\not\le S$. So let $SM\ne X$, so that $|X:SM|\ge p$. If $I$ is not a
prime, then $e(W)\ge p\cdot p^2\ge p\,\mathfrak q(p)$. Let $I=m$ be a prime,
and let $j$ be the unique index with $M_j\not\le S$. Then $H:=M_j\cap S$ is a
subgroup of index $m$ of $M_j$, and $H$ is normalized by $SM$. Let $\Omega$ be
the set of $X$-conjugates of $H$; then $|X:SM|\ge|X:N_X(H)|=|\Omega|$.
Moreover, $M$ acts trivially on $\Omega$, so that $X/M\cong G/G''$ acts
transitively on $\Omega$. By (c), some $R\in M_G$ in the first layer, with
$R\in\Syl_r(G)$ say, acts faithfully on $P_j$, and some member $Q$ of the
$0$-th layer, of prime order $q$, acts nontrivially on $R$. If $r$ divides
$m-1$, then $m\ge\mathfrak q(p)$ and $e(W)\ge|X:SM|\,m\ge p\,\mathfrak q(p)$.
So let $r\nmid m-1$. Then $R$ acts faithfully on $\Omega$: if $g\in R$ fixes
every $H'\in\Omega$, then $g$ induces on each $M_j/H'$, a group of order $m$,
an automorphism of $r$-power order, which is trivial; and as the intersection
of all $H'\in\Omega$ is a normal subgroup of $X$ properly contained in $M_j$,
hence trivial, $g$ centralizes $M_j$, whence $g=1$. The image of $R$ in
$X/M\cong G/G''$ is normal, being a Sylow subgroup of the abelian normal
subgroup $G'/G''$, and $Q$ normalizes $R$. Therefore Lemma~\ref{lem:orbit}
yields $|\Omega|\ge\mathfrak q(p)$, and so
$e(W)\ge|X:SM|\,m\ge|\Omega|\,m\ge p\,\mathfrak q(p)$.
\end{proof}

If the $0$-th layer is as large as possible, then the third layer can be
handled more efficiently.

\begin{lemma}\label{lem:rank3}
Let $p\ge 3$ be a prime, let $p^+$ be the smallest prime greater than $p$,
and let $G$ be a $\sigma$-reduced group with $\sigma(G)=n$ such that every
prime divisor of $|G|$ is at least $p$. If $n_0=n$, then every Sylow subgroup
of $G'''$ has rank at least $pp^+$.
\end{lemma}

\begin{proof}
By Lemma~\ref{lem:ranks}(1), it suffices to show that $e(S)\ge pp^+$ for
every $S\in M_G$ in the third layer.

So let $S\in M_G$ lie in the third layer, let $\ell$ be the prime with
$S\in\Syl_\ell(G)$, and let $V$ be the Sylow $\ell$-subgroup of
$G'''/G^{(4)}$. Then $V\cong S$, and $V$ is a $\GF(\ell)G$-module on which
$G'''$ acts trivially. The chief factors of $G$ of $\ell$-power order are the
composition factors of $V$; as $S$ is isomorphic to one of them, $V$ is
irreducible.

By (c) there are $U\in M_G$ in the second layer acting faithfully on $S$, and
$R_1\in M_G$ in the first layer acting faithfully on $U$; let $r$ be the
prime with $R_1\in\Syl_r(G)$. The group $\overline U:=UG'''/G'''\cong U$ is a
Sylow subgroup of the abelian normal subgroup $G''/G'''$ of $G/G'''$, so it is
normal in $G/G'''$, and it acts faithfully on $V$. Let $\overline{\mathbb F}$
be an algebraic closure of $\GF(\ell)$, and let $\Lambda$ be the set of linear
characters $\mu$ of $\overline U$ such that the $\mu$-eigenspace $V_\mu$ of
$\overline U$ on $V\otimes\overline{\mathbb F}$ is nonzero. Then
$e(S)=\dim V\ge|\Lambda|$. As $\overline U$ acts faithfully on $V$, the set
$\Lambda$ generates the character group of $\overline U$; hence every element
of $G$ which fixes all $\mu\in\Lambda$ centralizes $\overline U$. The group
$G$ acts on the characters of $\overline U$, with $G''$ acting trivially, and
since $V$ is irreducible, Clifford's theorem shows that $\Lambda$ is a single
orbit of the group $\Gamma:=(G/G'')\times\langle\varphi\rangle$, where
$\varphi$ maps $\mu$ to $\mu^\ell$. The image of $R_1$ in $G/G''$ is normal,
being a Sylow subgroup of the abelian normal subgroup $G'/G''$. Hence the
$R_1$-orbits on $\Lambda$ form a system of blocks for $\Gamma$, and they all
have the same length, $r^a$ say. As $R_1$ acts faithfully on $U$, it acts
faithfully on $\Lambda$, so that $a\ge 1$. Let $t$ be the number of
$R_1$-orbits on $\Lambda$; thus $|\Lambda|=tr^a$.

\smallskip
(i) \emph{Let $Q\in M_G$ lie in the $0$-th layer, let $|Q|=q$, and suppose
that $Q$ acts nontrivially on $R_1$. Then $|\Lambda|\ge\min\{qr,r^o\}$, where
$o$ is the order of $r$ modulo $q$.} This follows from Lemma~\ref{lem:orbit},
applied to $\Gamma$ acting on $\Lambda$, as $Q$ normalizes $R_1$.

\smallskip
(ii) Suppose that two distinct members $Q_1,Q_2$ of the $0$-th layer, of
orders $q_1,q_2$, act nontrivially on $R_1$, and assume that $e(S)<pp^+$. Let
$o_i$ be the order of $r$ modulo $q_i$. As $q_i$ and $r$ are distinct primes
at least $p$, we have $q_ir\ge pp^+$, so (i) yields
$r^{o_i}\le|\Lambda|\le e(S)<pp^+$. If $r=p$, then $o_i\ge 3$, since $q_i>p$
divides neither $p-1$ nor $p+1$ (the only divisor of $p+1$ exceeding $p$ is
$p+1$, which is even); but $p^3>pp^+$, as $p^+<2p$ by Bertrand's postulate.
Hence $r\ge p^+$, so that $r^2>pp^+$ and $o_1=o_2=1$. Thus $q_1q_2$ divides
$r-1$, and $e(S)\ge|\Lambda|\ge r>q_1q_2\ge pp^+$, a contradiction.

\smallskip
(iii) By (c), some member of the $0$-th layer acts nontrivially on $R_1$, so
by (ii) it remains to consider the case that exactly one member $Q_q$ of the
$0$-th layer, of order $q$, acts nontrivially on $R_1$. Let $Q^*$ be the
product of the other members of the $0$-th layer. As $n_0=n$ and $G/G'$ is
cyclic, $Q^*$ is cyclic of order the product of $n-1$ distinct primes; let
$z$ be a generator of $Q^*$. Then $Q^*$ centralizes $R_1$ and therefore
permutes the $R_1$-orbits on $\Lambda$. Suppose that $Q^*$ fixes every
$R_1$-orbit $O$. On $O$, the group $R_1$ induces a regular abelian permutation
group, and $Q^*$ induces a group of permutations commuting with it. Since the
centralizer of a regular abelian permutation group in the symmetric group is
the group itself, and $|Q^*|$ is prime to $r$, the group $Q^*$ acts trivially
on $O$. Hence $Q^*$ fixes every $\mu\in\Lambda$, so that it centralizes
$\overline U$, and as $Q^*$ normalizes $U$, we get $[Q^*,U]\le U\cap G'''=1$.
As the Sylow subgroups of $G$ for the primes dividing $|Q^*|$ have prime
order, $Q^*=\langle z\rangle$ is a central Hall subgroup of $C_G(z)$, and
therefore $C_G(z)=Q^*\times L$ for a normal Hall subgroup $L$ of $C_G(z)$.
By Lemma~\ref{lem:centralizer} we have $\sigma(L)\le 1$, that is, every
element of $L$ has prime power order, and so $\rho(L)\le 2$ by Higman's
theorem \cite{Hi57}. But $L$ contains $Q_q$, $R_1$ and $U$, a contradiction.
Therefore $Q^*$ moves some $R_1$-orbit. The length of its $Q^*$-orbit is a
nontrivial divisor of $|Q^*|$, hence divisible by a prime $s$ dividing
$|Q^*|$, so that $t\ge s$. As $s$ and $r$ are distinct primes at least $p$,
we conclude that $e(S)\ge|\Lambda|\ge sr\ge pp^+$.
\end{proof}

\begin{remark}\label{rem:rank3}
The hypothesis $n_0=n$ is used only in part (iii) of the proof of
Lemma~\ref{lem:rank3}, to obtain $\sigma(L)\le 1$. If $n_0=n-1$, the same
argument yields only $\sigma(L)\le 2$, which is compatible with
$\rho(L)\ge 3$, and we do not know whether the conclusion of
Lemma~\ref{lem:rank3} remains true in this case.
\end{remark}

We now collect the inequalities used in the computation.

\begin{proposition}\label{prop:search}
Let $p\ge 3$ be a prime, let $n\in\mathbb N$, and let $G\ne 1$ be a
$\sigma$-reduced group with $\sigma(G)\le n$ all of whose prime divisors are
at least $p$. Let $K:=\dl(G)-1$ and $\lambda_i:=n-n_i$ for $i\ge 0$. Then:
\begin{enumerate}[label=(\roman*)]
\item $1\le n_i\le n$ for $0\le i\le K$;
\item $K+1\le\min\{n+1+\lambda_0,\ n+2+\lambda_1,\ 2n\}$;
\item if $2\le s_1<\dots<s_m\le K$ with $s_{k+1}-s_k\ge 2$ for $1\le k<m$ and
$\sum_{i>s_1}n_i\le p-1$, then $n_{s_1}+\dots+n_{s_m}\le n$, and
$n_{s_1}+\dots+n_{s_m}+n_t\le n$ for every $t$ with $s_m<t\le K$;
\item if $2\le k_1<\dots<k_r\le K$, and $s_1,\dots,s_r$ are integers with
$1\le s_i\le n_{k_i}$ and $\prod_{l>i}(s_l+1)\le b_{k_i}$ for $1\le i\le r$,
then $s_1+\dots+s_r\le n$.
\end{enumerate}
Here $b_2:=p$, $b_3:=\mathfrak q(p)$, $b_4:=p\,\mathfrak q(p)$ and
$b_k:=\max\{p\,\mathfrak q(p),\lceil p^{(k-1)/2}\rceil\}$ for $k\ge 5$. If
$\lambda_0=0$, then (iv) holds with $b_k$ replaced by $\max\{b_k,pp^+\}$ for
$k\ge 3$.
\end{proposition}

\begin{proof}
(i) holds by \S\ref{sec:prelim}. For (ii), let $m:=\sigma(G)\le n$; then
Lemma~\ref{lem:dl} gives $\dl(G)\le m+1+(m-n_0)\le n+1+\lambda_0$ and
$\dl(G)\le m+2+(m-n_1)\le n+2+\lambda_1$, and $\dl(G)\le 2m\le 2n$ by
\cite[Lemma~3(c)]{Ke95a}. (iii) is Lemma~\ref{lem:chain}, since
$\rho(G^{(s_1+1)})=\sum_{i>s_1}n_i$. For (iv), choose $s_i$ members of $M_G$
in the $k_i$-th layer, for $1\le i\le r$, and let $U_i$ be their product. By
Lemma~\ref{lem:ranks}, $e(U_i)\ge b_{k_i}\ge\prod_{l>i}(s_l+1)$. If
$\lambda_0=0$, then $\sigma(G)=n_0=n$, and Lemmas~\ref{lem:rank3}
and~\ref{lem:ranks}(1) show that also $e(U_i)\ge pp^+$ whenever $k_i\ge 3$. As
members of $M_G$ in the same layer commute and members in earlier layers
normalize those in later layers (see \S\ref{sec:prelim}), $H:=U_1\cdots U_r$
is a subgroup of $G$ satisfying the hypotheses of Lemma~\ref{lem:counting}.
Hence $s_1+\dots+s_r=\rho(H)=\sigma(H)\le\sigma(G)\le n$.
\end{proof}

For a prime $p\ge 3$ and $n\in\mathbb N$, let $\beta(n,p)$ denote the maximum
of $\sum_{i=0}^Kn_i$ over all $K\ge 0$ and all integer vectors
$(n_0,\dots,n_K)$ satisfying (i)--(iv) of Proposition~\ref{prop:search},
including the refinement for $\lambda_0=0$. By Proposition~\ref{prop:search},
$B_{p-1}(n)\le\beta(n,p)$, so that \eqref{eq:reduction} yields
\begin{equation}\label{eq:beta}
\rho(n)\le\beta(n,p)+\pi(p-1).
\end{equation}
As $n_0$ and $n_1$ occur only in (i) and (ii), we have
\[
\beta(n,p)=\max_{0\le\lambda_0,\lambda_1<n}\Bigl(2n-\lambda_0-\lambda_1
+\max_{K\le\kappa(\lambda_0,\lambda_1)}F_p^{\lambda_0}(K)\Bigr),
\]
where $\kappa(\lambda_0,\lambda_1):=\min\{n+\lambda_0,\,n+1+\lambda_1,\,2n-1\}$,
and $F_p^{\lambda_0}(K)$ denotes the maximum of $n_2+\dots+n_K$ over all
$(n_2,\dots,n_K)$ satisfying (iii) and (iv), with the refined bounds $b_k$ if
$\lambda_0=0$ (and $F_p^{\lambda_0}(K):=0$ for $K\le 1$).

We compute $F_p^{\lambda_0}(K)$ by dynamic programming, treating the layers
from the $K$-th one upwards. Once the layers $j,\dots,K$ have been treated, the
constraints involving the remaining layers depend only on the following
data: the sum $n_j+\dots+n_K$, which decides whether a chain in (iii) may
start in an earlier layer; the maximum of $n_j,\dots,n_K$; the largest
left-hand sides in (iii) of chains starting in the $j$-th and in the
$(j+1)$-th layer, and the maximum of these quantities over all chains starting
in later layers; and the set of pairs $\bigl(\prod_l(s_l+1),\sum_l s_l\bigr)$
arising from the choices in (iv) with $k_1\ge j$ that are not dominated by
another such pair and whose first entry is at most $b_{j-1}$. For fixed values
of the other data, a larger sum $n_j+\dots+n_K$ only weakens the remaining
constraints, so it suffices to keep the largest one. The results are listed in
Table~3, where, for each $n$, $p$ is a prime with $n<p\le p_0$, $p_0$ the
smallest prime greater than $2n$, for which \eqref{eq:beta} gives the best
bound. In all cases, the maximum defining $\beta(n,p)$ is attained for
$(\lambda_0,\lambda_1)=(1,0)$, and for $n\in\{6,8,16,17\}$ also for
$(\lambda_0,\lambda_1)=(0,0)$. The values of $\beta(n,p)$ have been computed
independently by two different programs, one of which has also been checked
against an exhaustive search for small values of $n$; a Python implementation
is available from the authors.

\begin{center}
	\begin{tabular}{c|cccccccccccc}
		\hline $n$ & 6 & 7 & 8 & 9 & 10 & 11 & 12 & 13 & 14 & 15 & 16 & 17\\ \hline
		$p$ & 7 & 11 & 13 & 11 & 13 & 19 & 17 & 23 & 29 & 29 & 31 & 37\\
		$\beta(n,p)$ & 21 & 25 & 28 & 34 & 38 & 41 & 48 & 50 & 54 & 58 & 62 & 66\\
		$\pi(p-1)$ & 3 & 4 & 5 & 4 & 5 & 7 & 6 & 8 & 9 & 9 & 10 & 11\\ \hline
		$\rho(n)\le$ & 24 & 29 & 33 & 38 & 43 & 48 & 54 & 58 & 63 & 67 & 72 & 77\\ \hline
	\end{tabular}
	
	\medskip
	Table 3. The bounds of Theorem B.
\end{center}

\begin{proof}[Proof of Theorem B]
This follows from \eqref{eq:beta} and Table~3.
\end{proof}

\begin{remark}\label{rem:small}
(a) Without the refinement for $\lambda_0=0$, that is, without
Lemma~\ref{lem:rank3}, the same computation gives the bounds
$25,29,34,39,44,49,55,59,64,68,73,78$ for $6\le n\le 17$. If the bound
$e_4\ge p\,\mathfrak q(p)$ of Lemma~\ref{lem:ranks}(3) is replaced by the
weaker bound $e_4\ge p^2$, then, for a suitable, possibly different, choice
of $p$, the bounds of Table~3 remain unchanged, except for
$n\in\{9,10,11,13\}$, where they increase by one.

(b) The bottleneck is the second layer, whose members have rank only $p$ in
general, so that Lemma~\ref{lem:counting} gives little information there,
while a chain in (iii) starting in the second layer requires
$\rho(G''')\le p-1$. For $n=18$, the computation gives no improvement on
Theorem A for any prime $18<p\le 37$; we have not carried it out for larger
$n$.
\end{remark}

\section{A lower bound}\label{sec:lower}

For a group $G$, the \emph{prime graph} $\Gamma_G$ of $G$ is the graph with
vertex set $\vp(G)$ in which two distinct primes $p$ and $q$ are adjacent if
and only if $G$ contains an element of order $pq$. We denote by
$\overline{\Gamma}_G$ the complement of $\Gamma_G$. For a graph $F$ we
denote by $\alpha(F)$ its independence number, i.e., the maximal size of a
set of pairwise non-adjacent vertices of $F$, and we call $F$
\emph{$3$-colorable} if its vertex set is the union of three independent
sets. We write $u\sim v$ if two vertices $u$ and $v$ are adjacent.

\begin{lemma}[{\cite[Theorems~2.8 and 2.10]{GKLNS15}}]\label{lem:GKLNS}
A graph $F$ is isomorphic to $\overline{\Gamma}_G$ for some solvable group
$G$ if and only if $F$ is triangle-free and $3$-colorable.
\end{lemma}

We shall only use that every triangle-free $3$-colorable graph is of the form
$\overline{\Gamma}_G$, which is \cite[Theorem~2.8]{GKLNS15}. In the proof
given there, the group $G$ is constructed explicitly from cyclic groups of
prime order and suitable modules, by means of Dirichlet's theorem on primes
in arithmetic progressions; in particular, the classification of finite
simple groups is not used.

\begin{lemma}\label{lem:alpha}
For every group $G$ we have $\sigma(G)\le\alpha(\overline{\Gamma}_G)$.
\end{lemma}

\begin{proof}
Let $g\in G$. If $p\ne q$ are prime divisors of $o(g)$, then
$g^{o(g)/pq}$ has order $pq$. Hence the primes in $\vp(o(g))$ are pairwise
adjacent in $\Gamma_G$, that is, they form an independent set of
$\overline{\Gamma}_G$.
\end{proof}

By Lemmas~\ref{lem:GKLNS} and \ref{lem:alpha}, every triangle-free
$3$-colorable graph $F$ yields a solvable group $G$ with $\rho(G)=|V(F)|$
and $\sigma(G)\le\alpha(F)$. As the vertex set of $F$ is the union of three
independent sets, we always have $|V(F)|\le 3\alpha(F)$; so the best one can
hope for in this way are graphs with $3n$ vertices and independence number
$n$. We now construct such graphs for all $n\ge 4$.

For $n\in\mathbb N$ let $F_n$ be the graph with the $3n$ vertices
$t_i,m_i,b_i$ $(1\le i\le n)$, whose edges are the following:
\begin{enumerate}[label=(E\arabic*)]
\item $t_im_i$ and $m_ib_i$ for $1\le i\le n$;
\item $t_im_{i+1}$, $m_ib_{i+1}$ and $b_it_{i+1}$ for all odd $i<n$;
\item $t_ib_{i+1}$, $m_it_{i+1}$ and $b_im_{i+1}$ for all even $i<n$;
\item $t_ib_j$ for all $i,j$ with $|i-j|\ge 2$.
\end{enumerate}
The graphs $F_5$, $F_6$ and $F_7$ are shown in Figure~\ref{fig:F}, where the
edges in (E4) are not drawn. We call $C_i:=\{t_i,m_i,b_i\}$ the $i$-th
\emph{column} of $F_n$, and $T:=\{t_1,\dots,t_n\}$, $M:=\{m_1,\dots,m_n\}$
and $B:=\{b_1,\dots,b_n\}$ its \emph{rows}. It is straightforward to check
that the definition can be rephrased as follows: no edge joins two vertices
in the same row, and for $1\le a,c\le n$ we have
\begin{align*}
t_a\sim m_c &\iff a=c, \text{ or $a$ is odd and $|a-c|=1$};\\
m_a\sim b_c &\iff a=c, \text{ or $a$ is odd and $|a-c|=1$};\\
t_a\sim b_c &\iff |a-c|\ge 2, \text{ or $a$ is even and $|a-c|=1$}.
\end{align*}

\begin{figure}[ht]
\centering
\drawF{5}\par\smallskip
{\small (a) $F_5$}\par\bigskip
\drawF{6}\par\smallskip
{\small (b) $F_6$}\par\bigskip
\drawF{7}\par\smallskip
{\small (c) $F_7$}
\caption{The graphs $F_5$, $F_6$ and $F_7$. The vertices $t_i$, $m_i$ and $b_i$ form the $i$-th column. Not drawn are the edges $t_ib_j$ with $|i-j|\ge 2$.}\label{fig:F} \end{figure}

\begin{lemma}\label{lem:Fn}
Let $n\ge 4$. Then $F_n$ is triangle-free, the rows $T$, $M$ and $B$ are
independent sets, and $\alpha(F_n)=n$.
\end{lemma}

\begin{proof}
By definition, no edge of $F_n$ joins two vertices in the same row, so $T$,
$M$ and $B$ are independent sets; in particular, $\alpha(F_n)\ge n$.

Since the rows are independent, a triangle of $F_n$ would consist of vertices
$t_a$, $m_c$ and $b_e$. If $a$ is even, then $t_a\sim m_c$ forces $c=a$, and
then $m_a\sim b_e$ forces $e=a$; but $t_a\not\sim b_a$. If $a$ is odd, then
$t_a\sim b_e$ forces $|a-e|\ge 2$, and $t_a\sim m_c$ forces
$c\in\{a-1,a,a+1\}$. If $c=a$, then $m_a\sim b_e$ yields $|a-e|\le 1$, a
contradiction. If $c=a\pm 1$, then $c$ is even, so $m_c\sim b_e$ yields
$e=c$ and hence $|a-e|=1$, again a contradiction. Thus $F_n$ is
triangle-free.

Now let $S$ be an independent set of $F_n$. Since $t_i\sim m_i\sim b_i$, we
have $|S\cap C_i|\le 2$ for all $i$, with equality only if
$S\cap C_i=\{t_i,b_i\}$. If $|S\cap C_i|\le 1$ for all $i$, then $|S|\le n$.
It therefore suffices to show the following: if $S\cap C_d=\{t_d,b_d\}$ for
some $d$, then $|S\cap C_i|\le 1$ for all $i\ne d$, and $S\cap C_i=\emptyset$
for some $i$. Here and below, $C_i$ is understood to be empty if
$i\notin\{1,\dots,n\}$.

Suppose first that $d$ is even. As $t_d\sim b_c$ for all $c\ne d$, we have
$S\cap B=\{b_d\}$, and hence $|S\cap C_i|\le 1$ for all $i\ne d$. If
$t_a\in S$, then $t_a\not\sim b_d$, whence $a\in\{d-1,d,d+1\}$. Moreover,
$b_d\sim m_c$ for $c\in\{d-1,d,d+1\}$, as $d\pm 1$ is odd. Consequently,
\[
S\cap C_{d\pm 1}\subseteq\{t_{d\pm 1}\}\qquad\text{and}\qquad
S\cap C_{d\pm 2}\subseteq\{m_{d\pm 2}\}.
\]
As $n\ge 4$, we have $d+2\le n$ or $d-2\ge 1$. We treat the case $d+2\le n$;
in the other case, the same argument applies to $C_{d-1}$ and $C_{d-2}$. If
$t_{d+1}\notin S$, then $S\cap C_{d+1}=\emptyset$. If $t_{d+1}\in S$, then
$m_{d+2}\notin S$, since $d+1$ is odd and so $t_{d+1}\sim m_{d+2}$; thus
$S\cap C_{d+2}=\emptyset$.

If $d$ is odd, the same argument works with the roles of $t$ and $b$
interchanged: now $b_d\sim t_a$ for all $a\ne d$, so that $S\cap T=\{t_d\}$;
$t_d\not\sim b_c$ implies $c\in\{d-1,d,d+1\}$; $t_d\sim m_c$ for
$c\in\{d-1,d,d+1\}$; and $b_{d\pm 1}\sim m_{d\pm 2}$, as $d\pm 2$ is odd.
This completes the proof.
\end{proof}

\begin{proof}[Proof of Theorem E]
Let $n\ge 4$. By Lemma~\ref{lem:Fn}, the graph $F_n$ is triangle-free and
$3$-colorable, so by Lemma~\ref{lem:GKLNS} there exists a solvable group $G$
with $\overline{\Gamma}_G\cong F_n$. (For instance, one may take the group
constructed in the proof of \cite[Theorem~2.8]{GKLNS15} from the orientation
of $F_n$ in which all edges are directed from $T$ to $M\cup B$ and from $M$
to $B$; this orientation is then the Frobenius digraph of $G$ in the sense of
\cite{GKLNS15}.) Then $\rho(G)=3n$, and $m:=\sigma(G)\le\alpha(F_n)=n$ by
Lemmas~\ref{lem:alpha} and \ref{lem:Fn}. Let $C$ be a cyclic group whose
order is the product of $n-m$ distinct primes not dividing $|G|$. Then
$G\times C$ is solvable, $\sigma(G\times C)=\sigma(G)+n-m=n$ and
$\rho(G\times C)=3n+n-m\ge 3n$. Hence $\rho(n)\ge 3n$. The remaining
assertions follow from Table~1.
\end{proof}

\begin{remark}\label{rem:lower}
(a) If $G$ and $H$ are groups of coprime orders, then
$\overline{\Gamma}_{G\times H}$ is the disjoint union of
$\overline{\Gamma}_G$ and $\overline{\Gamma}_H$, and we have
$\sigma(G\times H)=\sigma(G)+\sigma(H)$ and $\rho(G\times H)=\rho(G)+\rho(H)$.
Correspondingly, the disjoint union of triangle-free $3$-colorable graphs is
triangle-free and $3$-colorable, and independence numbers add up under
disjoint unions. Since every $n\ge 4$ is a sum of elements of
$\{4,5,6,7\}$, Theorem E could therefore also be derived from the cases
$n\in\{4,5,6,7\}$ alone, by taking disjoint unions of the graphs $F_4,\dots,F_7$,
or equivalently direct products of groups of coprime orders. This is the
strategy suggested in \cite[\S3]{BKT23}, where the case $n=4$ is covered by
the groups of \cite[Example~1.9]{Ke94}.

(b) For $n\le 3$ we have $\alpha(F_n)>n$; for instance, $\{t_1,t_2,t_3,b_2\}$
is an independent set of $F_3$. In fact, for $n\le 3$ there is no
triangle-free graph with $3n$ vertices and independence number $n$, since the
Ramsey numbers $R(3,2)=3$, $R(3,3)=6$ and $R(3,4)=9$ (see \cite{GG55}) satisfy
$R(3,n+1)\le 3n$ for $n\le 3$. In view of Lemmas~\ref{lem:GKLNS} and
\ref{lem:alpha}, this is in accordance with the values $\rho(n)=3n-1$ for
$n\le 3$ in Table~1.
\end{remark}

\begin{remark}\label{rem:W}
The graphs $F_n$ are closely related to the graphs $W_k$ studied by Hovey,
Keller, Kou, Liu, Sivaraman and Wu in \cite{HKKLSW}. For $n\ge 3$ and
$k=2n-2$, the graph $W_k$ of \cite[Definition~6.1]{HKKLSW} has the vertices
$(1,i)$ for odd $i$, $(3,i)$ for even $i$, and $(2,i)$ for $1<i<n$, where
$1\le i\le n$; two vertices $(r,i)$ and $(s,j)$ are adjacent if and only if
either $|r-s|=1$ and $|i-j|\le 1$, or $\{r,s\}=\{1,3\}$ and $|i-j|\ge 2$. By
\cite[Lemma~6.6 and Theorem~6.2]{HKKLSW}, for every even $k\ge 18$ the graph $W_k$ has a
unique $3$-coloring up to permutation of the colors, and it is the complement
of a superminimal prime graph in the sense of \cite{HKKLSW}.

Identifying $(1,i)$, $(2,i)$ and $(3,i)$ with $t_i$, $m_i$ and $b_i$,
respectively, we see that $W_{2n-2}$ is the subgraph of $F_n$ induced on the
set
\[
\{t_i\mid i\ \text{odd}\}\cup\{b_i\mid i\ \text{even}\}\cup\{m_i\mid 1<i<n\}.
\]
Indeed, by the description of the edges of $F_n$ given before
Lemma~\ref{lem:Fn}, for odd $a$ and even $e$ we have $t_a\sim m_c$ if and only
if $|a-c|\le 1$, $m_c\sim b_e$ if and only if $|c-e|\le 1$, and $t_a\sim b_e$
if and only if $|a-e|\ge 2$; moreover, in both graphs no edge joins two
vertices in the same row. Note, however, that the three rows of $W_{2n-2}$
have $\lceil n/2\rceil$, $n-2$ and $\lfloor n/2\rfloor$ vertices,
respectively, whereas each row of $F_n$ has exactly $n$ vertices; it is this
property of $F_n$, together with Lemma~\ref{lem:Fn}, that is needed for
Theorem E.
\end{remark}

\section{Further remarks}\label{sec:further}

We now turn to arbitrary finite groups. The following estimate is implicit
in the proof of \cite[Theorem~3.3]{HY21}, which in turn follows the proof of
\cite[Theorem~A]{Mo06}.

\begin{proposition}\label{prop:HY}
Let $n\in\mathbb N$, let $\kappa_n\ge\max_{m\le n}\rho(m)$, and let $\xi_n$ be
an upper bound for $\rho(T)$, where $T$ ranges over the direct products of
nonabelian simple groups with $\sigma(T)\le n$. Then
\[
\rho(\Gamma)\le\kappa_n\bigl(\xi_n+\kappa_n+1\bigr)
\]
for every group $\Gamma$ with $\sigma(\Gamma)\le n$.
\end{proposition}

\begin{proof}
As the order of an element of a section of $\Gamma$ divides the order of each
of its preimages, $\sigma(\Lambda)\le\sigma(\Gamma)\le n$ for every section
$\Lambda$ of $\Gamma$. Hence $\rho(\Lambda)\le\kappa_n$ for every solvable
section $\Lambda$ of $\Gamma$, and $\rho(T)\le\xi_n$ for every section $T$ of
$\Gamma$ which is a direct product of nonabelian simple groups. In particular,
we may assume that $\Gamma$ is not solvable. As in \cite[proof of
Theorem~A]{Mo06}, we consider the series
\[
1=S_0\le R_1<S_1<R_2<S_2<\dots<R_m<S_m\le R_{m+1}=\Gamma ,
\]
where $R_{i+1}/S_i$ is the largest solvable normal subgroup of $\Gamma/S_i$
for $i\ge 0$, and $S_i/R_i$ is the socle of $\Gamma/R_i$ for $i\ge 1$; here
$m\ge 1$, and as $\Gamma/R_i$ has no nontrivial solvable normal subgroup,
every $S_i/R_i$ is a direct product of nonabelian simple groups. Moret\'o
shows that for every $i$ there is a prime $q_i$ dividing $|S_i/R_i|$ but not
$|\Gamma/S_i|\,|R_i|$, and he constructs a solvable subgroup $H$ of $\Gamma$
whose order is divisible by $q_1,\dots,q_m$ (see \cite[proof of
Theorem~A]{Mo06} and \cite[proof of Theorem~3.3]{HY21}). The primes $q_i$ are
pairwise distinct, since $q_j$ divides $|S_j/R_j|$, and hence $|\Gamma/S_i|$,
for $j>i$. Thus $m\le\rho(H)\le\kappa_n$. As every prime divisor of
$|\Gamma|$ divides the order of a factor of the series above, we conclude that
\[
\rho(\Gamma)\le\sum_{i=1}^m\bigl(\rho(S_i/R_i)+\rho(R_i/S_{i-1})\bigr)
+\rho(\Gamma/S_m)\le m(\xi_n+\kappa_n)+\kappa_n
\le\kappa_n\bigl(\xi_n+\kappa_n+1\bigr). \qedhere
\]
\end{proof}

With $\kappa_n=7n$ and $\xi_n=28n^3$, Proposition~\ref{prop:HY} is the
estimate obtained in the proof of \cite[Theorem~3.3]{HY21}, and with
$\xi_n=28n^3$ it is \cite[(13)]{BKT23}; here $\xi_n=28n^3$ is admissible by
\cite[Theorem~3.2]{HY21}, which states that $\rho(T)<28\,\sigma(T)^3$ for every
direct product $T$ of nonabelian simple groups. For simple groups $T$, the
arguments of \cite{HY21} would in fact give $\rho(T)\le 8(\sigma(T)+1)^3+1$,
provided that the inequalities derived from \cite[Lemma~2.2]{HY21} hold for
the families of sets to which this lemma is applied (see
Remark~\ref{rem:HY}): for
the groups of Lie type, the proofs of \cite[Lemmas~2.4--2.6 and
Theorem~2.7]{HY21} yield one of the bounds $2+2\sigma(T)$,
$1+8\sigma(T)^3$, $1+8(\sigma(T)+1)^3$ and $18+17\sigma(T)$; the alternating
groups satisfy $\rho(A_k)\le 13\,\sigma(A_k)^2$ by \cite[Lemma~3.1]{HY21}; and
$\rho(T)\le 15$ for the sporadic groups and the Tits group. Under the
hypothesis of Corollary D we may therefore take $\xi_n=8(n+1)^3+1$, and
Proposition~\ref{prop:HY} yields
\begin{equation}\label{eq:HY}
\rho(\Gamma)\le\kappa_n\bigl(8(n+1)^3+\kappa_n+2\bigr)
\qquad\text{whenever }\sigma(\Gamma)\le n .
\end{equation}

\begin{remark}\label{rem:HY}
We have not been able to verify two steps in the proof of
\cite[Theorem~3.2]{HY21}, and therefore state Corollary D conditionally.

(a) The combinatorial \cite[Lemma~2.2]{HY21}, which is used for the classical
groups in \cite[Lemmas~2.4--2.6]{HY21}, is not correct as stated: for $n=2$
and the sets $A_1=\{x\}$, $A_2=\{y\}$ with $x\ne y$, its left hand side equals
$1$, while its right hand side equals $2^{1/3}$. (Its reduction to
\cite[Lemma~2.1]{HY21} requires the indices $a_1,\dots,a_k$ to be distinct,
whereas the proof of \cite[Lemma~2.1]{HY21} uses repeated indices.) The
conclusion of \cite[Lemma~2.2]{HY21} also fails for some of the families of
sets to which it is applied: for $S=\mathrm{PSL}_3(2)$, the inequality
$\rho(S)\le 1+\sigma(\mathrm{GL}_3(2))^3$ obtained in the proof of
\cite[Lemma~2.4]{HY21} is false, as $\rho(S)=3$ and
$\sigma(\mathrm{GL}_3(2))=\sigma(S)=1$. The final bound
$\rho(S)\le 1+8\sigma(S)^3$ holds in this case, and we do not know of a simple
group violating the bounds stated in \cite{HY21}.

(b) The case of direct products in \cite[Theorem~3.2]{HY21} is reduced to the
case of simple groups without further argument. This reduction is not
automatic: for $T=T_1\times\dots\times T_r$, the primes of $|T_j|$ which occur
in elements of $T_j$ with many prime divisors may also divide $|T_l|$ for
$l\ne j$, so that $\sigma(T)$ may be smaller than the sum of the
$\sigma(T_j)$. Combining the bounds for the factors directly only gives a
bound for $\rho(T)$ of order $\sigma(T)^4$. The socles $S_i/R_i$ in the proof
of Proposition~\ref{prop:HY} are such direct products, so this case is
needed.

Both points affect \cite[Theorem~1.1]{HY21} and \cite[Corollary~1]{BKT23} in
the same way as Corollary D. Unconditionally, Proposition~\ref{prop:HY}
reduces the problem of bounding $\rho_g(n)$ to that of bounding $\rho(T)$ for
direct products $T$ of nonabelian simple groups in terms of $\sigma(T)$.
\end{remark}

We put $\Theta_n:=\kappa_n\bigl(8(n+1)^3+\kappa_n+2\bigr)n^{-4}$, and we choose $\kappa_n=\rho(n)$ for $2\le n\le 4$ (see Table~1), $\kappa_5=18$ (by \cite{Ke95b}), $\kappa_n$ as in Theorem B for $6\le n\le 17$, and $\kappa_n=4n-1+\pi(2n)$ for $n\ge 18$ (by Theorem A); these values are increasing in $n$. This gives the following values.

\begin{center}
	\begin{tabular}{c|ccccccccccc}
		\hline $n$ & 2 & 3 & 4 & 5 & 6 & 7 & 8 & 9 & 10 & 11 & 12\\ \hline
		$\kappa_n$ & 5 & 8 & 12 & 18 & 24 & 29 & 33 & 38 & 43 & 48 & 54\\
		$\Theta_n\le$ & 69.7 & 51.6 & 47.6 & 50.4 & 51.3 & 49.9 & 47.3 & 46.6 & 46.0 & 45.5 & 46.0\\ \hline
	\end{tabular}
	
	\medskip
	Table 4. Values of $\Theta_n=\kappa_n\bigl(8(n+1)^3+\kappa_n+2\bigr)n^{-4}$, rounded up.
\end{center}

\begin{proof}[Proof of Corollary D]
For $n=1$ the assertion holds unconditionally, since $\rho_g(1)=4$ by
\cite{Hi57} and \cite{Br81}. Let $n\ge 2$, and assume the hypothesis of
Corollary D. Then \eqref{eq:HY} yields $\rho_g(n)\le\Theta_n n^4$,
and $\Theta_n<69.7$ for $2\le n\le 7$ by Table~4. For $n\ge 8$ we have
$\kappa_n\le\frac{101}{22}n$ by the proof of Corollary C, and therefore
\[
\Theta_n\le\frac{101}{22}\Bigl(8\Bigl(1+\frac1n\Bigr)^3+\frac{101}{22n^2}
+\frac{2}{n^3}\Bigr)
\le\frac{101}{22}\Bigl(8\cdot\Bigl(\frac98\Bigr)^3+\frac{101}{1408}
+\frac{1}{256}\Bigr)<52.7 .
\]
Hence $\rho_g(n)\le 70\,n^4$ for all $n\in\mathbb N$.
\end{proof}

Note that the maximum in Table~4 is attained for $n=2$; for $n\ge 8$, the
values of $\Theta_n$ are smaller than $52.7$.

\begin{remark}\label{rem:range}
Let $17\le n\le 43$ or $n\in\{45,55,57\}$. For these $n$, Theorems A and B
do not yield $\rho(n)\le\frac92 n$; they yield
\[
\rho(n)\le\Bigl\lfloor\frac92 n\Bigr\rfloor+\delta_n,\qquad\text{where }
\delta_n\in\{1,2\},
\]
and $\delta_n=2$ exactly for $n\in\{19,21,22,23,24,25,27,31,37\}$. As $\pi(2n)-\lfloor\frac n2\rfloor\le 3$
for all $n$, \eqref{eq:reduction} shows that the bound $\rho(n)\le\frac92n$
would follow for these $n$ if one could show that $\rho(G)\le 4n-3$ for every
$\sigma$-reduced group $G\in\cG_{2n}$ with $\sigma(G)=n$. Since
$\rho(G)\le 4n-1-\lambda_0-\lambda_1$ by Proposition~\ref{prop:key2}, it would
suffice to treat the case $\lambda_0+\lambda_1\le 1$, in which
$\dl(G)\le n+2$ by Lemma~\ref{lem:dl}. In this situation, the computation of
\S\ref{sec:small} with the smallest prime $p>2n$ uses all the tools at hand:
by Proposition~\ref{prop:key}, Lemma~\ref{lem:chain} applies to every
sequence, and the ranks are bounded by Lemma~\ref{lem:ranks} and, if
$\lambda_0=0$, by Lemma~\ref{lem:rank3}. For $n=17$ this yields Theorem B,
but for $n=18$ it gives no improvement on Theorem A
(Remark~\ref{rem:small}). Further progress therefore requires additional
structural information on the groups with $\lambda_0+\lambda_1\le 1$, such as
an extension of Lemma~\ref{lem:rank3} to the case $n_0=n-1$ (see
Remark~\ref{rem:rank3}).
\end{remark}

By \cite{Ke95c}, $\rho(n)\le(4+\varepsilon)n$ for every $\varepsilon>0$ and
all sufficiently large $n$, and Theorem A gives $\rho(n)\le 4n-1+\pi(2n)$
for all $n$. This naturally raises the following questions.

\begin{question}\label{q:range}
Is it true that $\rho(n)\le\frac92 n$ for every $n\in\mathbb N$? By
Theorems A and B, only the values $17\le n\le 43$ and $n\in\{45,55,57\}$
remain open.
\end{question}

\begin{question}\label{q:Gn}
Let $G$ be a $\sigma$-reduced group with $\sigma(G)=n$ such that every prime
divisor of $|G|$ is greater than $n$. Is it true that
$\rho(G'')=\nu_2+\nu_3\le 2n-1$, as in Proposition~\ref{prop:key2}? By
\eqref{eq:reduction}, a positive answer would yield
$\rho(n)\le 4n-1+\pi(n)$ for all $n$, and hence $\rho(n)\le\frac92 n$ for
all $n$, since $\pi(n)\le\frac n2+1$ (all primes but $2$ are odd).
\end{question}

\begin{question}\label{q:4n}
Is it true that $\rho(n)\le 4n$ for every $n\in\mathbb N$?
\end{question}

We recall that the conjectured optimal bound is $\rho(n)\le 3n$, in analogy
with Huppert's $\rho$--$\sigma$ conjecture for character degrees. By
Theorem E, this bound would be best possible for every $n\ge 4$.

\bigskip
\noindent\textbf{Acknowledgements.} We are grateful to Chiara Bellotti and
Timothy Trudgian for helpful discussions concerning the proof in
\cite{BKT23}, and to Benjamin Keller for first constructing the graphs $F_5$, $F_6$, and $F_7$. The second author was supported by a grant from the Simons Foundation (SFI-MPS-TSM-00014153 to Thomas M. Keller).

\bigskip
\noindent\textbf{AI statement.}
This project was started almost two years ago when the first author noticed the mistake in \cite{BKT23}. The results in the paper are the authors', but we used Claude by Anthropic to synthesize various bits and pieces into one homogeneous paper and to help with the computational work;
along the way Claude also noticed a couple of inaccuracies and fixed them. After using this tool, the authors reviewed and edited the content as needed and take full responsibility for the final mathematical and narrative integrity of the paper.

\bigskip
\noindent Enrico Jabara\\
Free Researcher\\
Castello 5135, 30122 Venezia, Italy\\
E-mail: \texttt{jabaraenrico@gmail.com}

\medskip
\noindent Thomas Michael Keller\\
Department of Mathematics, Texas State University\\
601 University Drive, San Marcos, TX 78666, USA\\
E-mail: \texttt{keller@txstate.edu}


\begin{thebibliography}{99}

\bibitem{BKT23}
C.~Bellotti, T.~M.~Keller and T.~S.~Trudgian, New bounds for numbers of
primes in element orders of finite groups, \emph{Math. Nachr.} \textbf{296}
(2023), no.~11, 5227--5231.

\bibitem{Br81}
R.~Brandl, Finite groups all of whose elements are of prime power order,
\emph{Boll. Un. Mat. Ital. A (5)} \textbf{18} (1981), no.~3, 491--493.

\bibitem{GG55}
R.~E.~Greenwood and A.~M.~Gleason, Combinatorial relations and chromatic
graphs, \emph{Canad. J. Math.} \textbf{7} (1955), 1--7.

\bibitem{GKLNS15}
A.~Gruber, T.~M.~Keller, M.~L.~Lewis, K.~Naughton and B.~Strasser, A
characterization of the prime graphs of solvable groups, \emph{J. Algebra}
\textbf{442} (2015), 397--422.

\bibitem{Hi57}
G.~Higman, Finite groups in which every element has prime power order,
\emph{J. London Math. Soc.} \textbf{32} (1957), 335--342.

\bibitem{HKKLSW}
D.~Hovey, T.~M.~Keller, J.~Kou, E.~Liu, R.~Sivaraman and E.~Wu, Properties and
families of minimal prime graphs, \emph{Ars Math. Contemp.}, to appear,
doi:10.26493/1855-3974.3553.b20.

\bibitem{HY21}
N.~N.~Hung and Y.~Yang, On the prime divisors of element orders,
\emph{Math. Nachr.} \textbf{294} (2021), no.~10, 1905--1911.

\bibitem{Hu67}
B.~Huppert, \emph{Endliche Gruppen I}, Springer-Verlag, Berlin, Heidelberg,
New York, 1967.

\bibitem{Ke94}
T.~M.~Keller, Solvable groups with a small number of prime divisors in the
element orders, \emph{J. Algebra} \textbf{170} (1994), no.~2, 625--648.

\bibitem{Ke95a}
T.~M.~Keller, On the asymptotic behaviour of $\rho(n)$, \emph{J. Algebra}
\textbf{174} (1995), no.~2, 587--598.

\bibitem{Ke95b}
T.~M.~Keller, Solvable groups with at most four prime divisors in the element
orders, \emph{J. Algebra} \textbf{175} (1995), no.~1, 1--23.

\bibitem{Ke95c}
T.~M.~Keller, A linear bound for $\rho(n)$, \emph{J. Algebra} \textbf{178}
(1995), no.~2, 643--652.

\bibitem{Mo06}
A.~Moret\'o, On the number of different prime divisors of element orders,
\emph{Proc. Amer. Math. Soc.} \textbf{134} (2006), no.~3, 617--619.

\bibitem{Mo21}
A.~Moret\'o, Huppert's conjecture for character codegrees, \emph{Math.
Nachr.} \textbf{294} (2021), no.~11, 2232--2236.

\bibitem{RS62}
J.~B.~Rosser and L.~Schoenfeld, Approximate formulas for some functions of
prime numbers, \emph{Illinois J. Math.} \textbf{6} (1962), 64--94.

\bibitem{Th59}
J.~G.~Thompson, Finite groups with fixed-point-free automorphisms of prime
order, \emph{Proc. Nat. Acad. Sci. U.S.A.} \textbf{45} (1959), 578--581.

\bibitem{Tu84}
A.~Turull, Fitting height of groups and of fixed points, \emph{J. Algebra}
\textbf{86} (1984), no.~2, 555--566.

\bibitem{Ya21}
Y.~Yang, On analogues of Huppert's conjecture, \emph{Bull. Aust. Math. Soc.}
\textbf{104} (2021), no.~2, 272--277.

\bibitem{Zh95}
J.~P.~Zhang, Arithmetical conditions on element orders and group structure,
\emph{Proc. Amer. Math. Soc.} \textbf{123} (1995), no.~1, 39--44.

\end{thebibliography}
\end{document}